\documentclass[english]{m2an}
\usepackage[T1]{fontenc}
\usepackage[latin9]{inputenc}
\usepackage{array}
\usepackage{mathtools}
\usepackage{multirow}
\usepackage{amsmath}
\usepackage{amsthm}
\usepackage{amssymb}
\usepackage{graphicx}
\usepackage{bm}
\usepackage{stmaryrd}

\usepackage{xcolor}
\usepackage[colorlinks,
linkcolor=blue,
anchorcolor=yellow,
citecolor=red,
]{hyperref}
\allowdisplaybreaks

\makeatletter

\theoremstyle{plain}
\ifx\thesection\undefined
\newtheorem{Scheme}{\protect\schemename}
\else

\fi
\theoremstyle{plain}
\ifx\thesection\undefined
\newtheorem{thm}{\protect\theoremname}
\else
\newtheorem{thm}{\protect\theoremname}[section]
\fi
\theoremstyle{remark}
\ifx\thesection\undefined
\newtheorem{rem}{\protect\remarkname}
\else
\newtheorem{rem}{\protect\remarkname}[section]
\fi
\theoremstyle{plain}
\ifx\thesection\undefined
\newtheorem{lem}{\protect\lemmaname}
\else

\fi
\theoremstyle{definition}
\ifx\thesection\undefined
\newtheorem{example}{\protect\examplename}
\else

\fi
\theoremstyle{definition}
\ifx\thesection\undefined
\newtheorem{prop}{\protect\propname}
\else
\newtheorem{prop}{\protect\propname}[section]
\fi
\@ifundefined{showcaptionsetup}{}{%
	\PassOptionsToPackage{caption=false}{subfig}}
\usepackage{subfig}
\makeatother

\usepackage{babel}
\providecommand{\examplename}{Example}
\providecommand{\schemename}{Scheme}
\providecommand{\lemmaname}{Lemma}
\providecommand{\remarkname}{Remark}
\providecommand{\theoremname}{Theorem}
\providecommand{\propname}{Proposition}

\begin{document}
\title{The virtual element method for the three dimensional inductionless magnetohydrodynamic model}

\author{Xianghai Zhou}
\address{College of Mathematics and System Sciences, Xinjiang University, Urumqi 830046, China. Email: zxhmath166@163.com.}

\author{Haiyan Su}
\address{Corresponding author. College of Mathematics and System Sciences, Xinjiang University, Urumqi 830046, China. Email: shymath@126.com.}


\begin{abstract}
This paper proposes a novel first-order and a novel second-order fully discrete virtual element schemes based on the scalar auxiliary variable method for the three dimensional inductionless magnetohydrodynamics problem. The backward Eular formula and the backward differential formula are used for the time discretization and two types conservation virtual element formulations are employed for spatial discretization. The main advantages include that the mass conservation in the velocity field and the charge conservation in the current density field are kept by taking characteristics of the virtual element method's discrete complex structures, the nonlinear term is handled explicitly by applying the scalar auxiliary variable method, the current density field is decoupled from the momentum equation, and the velocity field is decoupled from the Ohm's law. The unconditionally stable of the two fully discrete schemes are demonstrated. Finally, we present numerical experiment to verify the valid of the proposed schemes.
\end{abstract}
%
%
%
\maketitle
\section{Introduction}
Magnetohydrodynamics (MHD) is a branch of physics that studies the motion laws of electrically conducting fluids under the influence of a magnetic field. MHD has been widely applied in various fields such as liquid metal magnetic pumps, cooling of liquid metals in fusion reactors, aluminum electrolysis, magnetohydrodynamic engines, electromagnetic stirring, and geodynamics \cite{roberts1967introduction,cowling1976magnetohydrodynamics}. However, when the magnetic Reynolds number is small, the induced magnetic field can be disregarded compared to the external magnetic field, or it may reach saturation, in which case the magnetic field $\boldsymbol{B}$ can be treated as known \cite{ni2007currentI,ni2007currentII}. Other relevant models can be found in \cite{badia2014block,long2022convergence,wang2024decoupled}.

Due to its wide range of applications, the research on effective numerical methods for approximating solution to the MHD equations has been a focal point of study. The research outcomes in this field are abundant, and a comprehensive summary would exceed the scope of this paper.
Therefore, here we only make a simple classification and comment on these contributions. Because the critical nature (and limitations) of these schemes help inspire the introduction of our new approach. A type stabilized finite element method is proposed to deal with the convective dominance and compatibility conditions of the inductionless MHD problem in \cite{planas2011approximation}. Some finite element methods based on charge-conservation solve the incompressible inductionless MHD problem are proposed in \cite{li2019charge,li2019chargeR,zhang2021coupled,zhang2022decoupled,zhang2022fully}. In \cite{zhou2024local}, the authors propose local and parallel finite element methods based on charge-conservation to deal with large linear algebraic systems arising from solving the stationary inductionless MHD problem.

In \cite{shen2018scalar}, Shen et al. propose the SAV method for the gradient flow and its attraction lies in the numerical approximation of nonlinear systems.  The main reason is that this method allows us to design numerical schemes that are unconditionally energy stable and require only the solution of a decoupled linear system with constant coefficients at each time step. Due to its efficiency, flexibility and accuracy, it has been a powerful approach to develop energy stable numerical schemes for general dissipative systems, such as Navier-Stokes problem \cite{lin2019numerical,li2022new}, Cahn-Hilliard-Navier-Stokes problem \cite{yang2021novel,li2022fully,li2023consistency} and MHD problem \cite{yang2021second,li2022stability,dong2024optimal}.
In \cite{zhang2024stability}, the authors present stability and error analysis of the scalar auxiliary variable (SAV) schemes for the inductionless MHD equations.

The VEM method can be viewed as an extension of the FEM and was originally introduced in \cite{beirao2013basic}, and its practical implementation details were introduced in \cite{beirao2014hitchhiker}. The main characteristics of the VEM method include the ability of handling very general polygonal/polyedral meshes, the feasibility of guaranteeing physical properties, the possibility of easily implementing highly order approximation and highly regular discrete spaces, and it avoids an explicit expression of a local basis function and includes (but is not limits to) standard polynomials (only need to define appropriate the degrees of freedom and some operators involved in the discretization of the problem). These characteristics lead to its application to a wide range of problems. Here, we only exhibit some contributions related to this paper; for details, see\cite{brezzi2014basic,da2016h,caceres2017mixed,da2018family,da2018virtual,beirao2020stokes,alvarez2021virtual,naranjo2023virtual,beirao2023virtual,zhou2023full,dong2024virtual}.
In \cite{zhou2023full}, the authors  propose a full divergence-free of high order virtual finite element method to approximation of stationary inductionless MHD equations on polygonal meshes. However, as far as is known, VEM has made less contribution to dealing with three-dimensional problems, especially for MHD problems.

The aim of the present paper is to construct a novel first-order and a novel second-order fully discrete virtual element approximations of the three dimensional incompressible inducitonless MHD equations, which easier to obtain a fully divergence-free numerical approximation scheme than standard FEMs and thus guarantee the physical properties of the model. Moreover, the nonlinear term is handled explicitly based on the SAV method. The current density field is decoupled from the momentum equation, and the velocity field is decoupled from Ohm's law. Thus, we only need to solve the Stokes and Mixed-Poisson subproblems with constant coefficients at each time step, which ensure the high efficiency and unconditional stability of schemes. The backward Eular formula and the backward differential formula are used for the time discretization, and two types conservation virtual element formulations (employing the enhanced Stokes-like virtual element \cite{beirao2020stokes} to approximate the velocity, discontinuous piecewise polynomials for the pressure and electric potential, and $\boldsymbol{H}(\text{div},\mathrm{\Omega})$-conforming virtual element for the current density) are employed for spatial discretization.

The paper is organized as follows. In Section 2, we present the SAV scheme and its the weak formulation. In Section 3, we introduce the virtual element spaces for velocity, current density, pressure, and electric potential. Furthermore, we provide definitions of discrete bilinear and trilinear forms, and summarize some related properties. In Section 4, we propose a novel first-order and a novel second-order fully discrete virtual element approximations based on SAV method for the three dimensional incompressible inducitonless MHD equations and establish the uncondition stability estimates. In Section 5, we present some numerical experiments. Finally, conclusions are drawn in Section 6.
\subsection{Notations and continuous spaces}
Let $\mathrm{G}\subset\mathbb{R}^{3}$ be an open, bound and connected subset, we introduce some Sobolev spaces \cite{li2006finite}. Let $W^{m,p}(\mathrm{G})$ denote the standard Sobolev space equipped with norm $\|\cdot\|_{m,p}$ for $m\in\mathbb{N}^{+}$, $1\leq p\leq\infty$. For $m=0$,
$L^{p}(\mathrm{G})$ is the space of $p$-integrable functions with norm denoted by $\|\cdot\|_{L^{p}(\mathrm{G})}$ and equip with norm $\|u\|_{0,\mathrm{G}}:=\|u\|_{L^{2}(\mathrm{G})}$.
For $p = 2$, we write $H^{m}(\mathrm{G})$ for $W^{m,2}(\mathrm{G})$ and equip with norm $\|\cdot\|_{m}$ and semi-norm $|\cdot|_{m}$. To simplify, introducing some notations for function spaces
\begin{align*}
H_{0}^{1}(\mathrm{G})=\Big\{v\in H^{1}(\mathrm{G}),~v=0, \text{ on }\partial\mathrm{G}\Big\},\quad L_{0}^{2}(\mathrm{G})=\Big\{q\in L^{2}(\mathrm{G}), \int_{\mathrm{G}}q\mathrm{d}x=0\Big\},\\
\boldsymbol{H}_{0}(\text{div},\mathrm{G})=\Big\{\boldsymbol{v}\in [L^{2}(\mathrm{G})]^{3},\text{div}\boldsymbol{v}\in L^{2}(\mathrm{G}),~\boldsymbol{v}\cdot\boldsymbol{n}=0,\text{ on }\partial\mathrm{G}\Big\},
\end{align*}
and setting `` $\mathrm{G}=\mathrm{\Omega}$ '', we define
\begin{equation*}
\boldsymbol{U}:=[H_{0}^{1}(\mathrm{\Omega})]^{3},\quad\boldsymbol{E}:=\boldsymbol{H}_{0}(\text{div},\mathrm{\Omega}),\quad Q:=L_{0}^{2}(\mathrm{\Omega}),\quad \Psi:=L_{0}^{2}(\mathrm{\Omega}).
\end{equation*}

\subsection{The time dependent inductionless MHD model\label{IMHD system}}
In this paper, we consider the three dimensional time dependent incompressible inductionless MHD
equations as follows \cite{ni2007currentII}:
\begin{subequations}
	\label{model:IMHD}
	\begin{align}
		\boldsymbol{u}_{t}-R_{e}^{-1}\Delta\boldsymbol{u}+(\nabla\boldsymbol{u})\boldsymbol{u}+\nabla p - \kappa\boldsymbol{J}\times\boldsymbol{B} =\boldsymbol{0}&\qquad\text{ in }\mathrm{\Omega}\times(0,\mathrm{T}],\label{model:u}\\
\boldsymbol{J}+\nabla \phi - \boldsymbol{u}\times\boldsymbol{B} =\boldsymbol{0}&\qquad\text{ in }\mathrm{\Omega}\times(0,\mathrm{T}],\label{model:J}\\		\mathrm{div}\boldsymbol{u}  =0,\quad\mathrm{div}\boldsymbol{J} =0&\qquad\text{ in }\mathrm{\Omega}\times(0,\mathrm{T}],\label{model:divu div J}\\
		\boldsymbol{u}(\boldsymbol{x},0)  =\boldsymbol{u}_{0}(\boldsymbol{x})&\qquad\text{ in }\mathrm{\Omega},\label{model:uinit}\\
		\boldsymbol{u}  =\boldsymbol{g}&\qquad\text{ on }\mathrm{\Gamma}_{d}\times(0,\mathrm{T}],\label{model:ubc1}\\
R_{e}^{-1}\frac{\partial\boldsymbol{u}}{\partial\boldsymbol{n}}-p\boldsymbol{n}=0&\qquad\text{ on }\mathrm{\Gamma}_{n}=\mathrm{\Gamma}_{\mathrm{\Omega}}\setminus\mathrm{\bar{\Gamma}}_{d}\times(0,\mathrm{T}],\label{model:ubc2}\\
		\boldsymbol{J}\cdot\boldsymbol{n}  =0&\qquad\text{ on }\mathrm{\Gamma}_{i}\times(0,\mathrm{T}],\label{model:ubc3}\\
		\phi  =\Xi&\qquad\text{ on }\mathrm{\Gamma}_{c}=\mathrm{\Gamma}_{\mathrm{\Omega}}\setminus\mathrm{\bar{\Gamma}}_{i}\times(0,\mathrm{T}],\label{model:ubc3}
	\end{align}
\end{subequations}
where $\Omega\subset\mathbb{R}^{3}$ is a bounded domain with
Lipschitz-continuous boundary $\Gamma\coloneqq\partial\Omega$,
$T>0$ is the final time, $\boldsymbol{n}$ is the unit outer normal to $\mathrm{\Omega}$. The initial condition $\boldsymbol{u}_{0}\in[H^{1}(\mathrm{\Omega})]^3$ satisfies $\text{div} \boldsymbol{u}_{0}=0$. The unknowns are the
velocity $\boldsymbol{u}$, the current density $\boldsymbol{J}$, and the pressure $p$, and the electric potential $\phi$.
The parameters $R_{e}>0$ and $\kappa>0$ are the the Reynolds number and the coupling number between the fluid and the current density, respectively. The parameters are given by $R_{e} = \rho Lu_{0}/\nu$ and $\kappa = \sigma LB_{0}^{2}/(\rho u_{0})$ (see \cite{ni2007currentI}). Generally, $\Gamma_{d}$ denotes the inflow boundary, $\Gamma_{n}$ denotes the outflow boundary,
$\Gamma_{i}$ denotes the insulating boundary, $\Gamma_{c}$ denotes the conductive boundary, please refer to \cite{ni2007currentI,ni2007currentII} for the more details, and where $(\boldsymbol{g},\Xi)\in \boldsymbol{H}^{1/2}(\Gamma_{d})\times H^{1/2}(\Gamma_{c})$. For the sake of simplicity, only theoretical analyses under homogeneous Dirichlet boundary conditions
\begin{equation}
\boldsymbol{u}=\boldsymbol{0},\quad
\boldsymbol{J}\cdot \boldsymbol{n}=0\quad\text{on}\;\mathrm{\Gamma_\Omega}\times(0,\mathrm{T}],
\label{Boundary:Dirichlet}
\end{equation}
are considered in this paper.

The time dependent inductionless MHD model is an energy dissipative
system. Namely, we take the inner products of (\ref{model:u}) with $\boldsymbol{u}$ and (\ref{model:J}) with $\boldsymbol{J}$, respectively, and then combine these with (\ref{model:divu div J}) to obtain the following energy dissipation law
\[
\frac{{\rm d}}{{\rm d}t}{\rm E}_{\text{IMHD}}(\boldsymbol{u})=-R_{e}^{-1}\left\Vert \nabla\boldsymbol{u}\right\Vert_{0,\mathrm{\Omega}}^{2}-\kappa\left\Vert \boldsymbol{J}\right\Vert_{0,\mathrm{\Omega}}^{2}\text{ with }{\rm E}_{\text{IMHD}}(\boldsymbol{u})=\frac{1}{2}\left\Vert \boldsymbol{u}\right\Vert_{0,\mathrm{\Omega}}^{2}.
\]
The energy dissipation law is a very important property for this model
in physics and mathematics. Thus, it is preferred to propose the
numerical schemes which preserve
a dissipative energy law at the discrete level.

Here, we use the fact that is
\begin{equation}
\big((\nabla\boldsymbol{u})\boldsymbol{u},\boldsymbol{u}\big)=0,\quad \kappa(\boldsymbol{J}\times\boldsymbol{B},\boldsymbol{u})-\kappa(\boldsymbol{u}\times\boldsymbol{B},\boldsymbol{J})=0.
\label{eq:zero}
\end{equation}
These features are also called \textquotedblleft zero-energy-contribution\textquotedblright ,
which play an important role in designing efficient scheme.

\section{an equivalent system\label{continuous}}

\subsection{The SAV scheme for the inductionless equations\label{sec:SAV schemes IMHD}}
Inspired by \cite{yang2021novel,li2022new}, we introduce scalar auxiliary
variable $s(t)$ as
\begin{equation}
	s(t)=R(t)=\exp\left(-t/T\right),\label{eq:SAV}
\end{equation}
By taking the derivative of (\ref{eq:SAV}) with respect to $t$,
we obtain
\[
\frac{\mathrm{d}s}{\mathrm{dt}}=-sQ,\quad\text{with}\quad Q=\frac{1}{T}.
\]
Observed that this is a linear and dissipative or conservative ordinary
differential equation for the SAV. This feature
is vital for designing unconditionally energy-stable and linear decoupled
scheme.
In light of equation $s(t)/R(t)=1$ and (\ref{Boundary:Dirichlet})-(\ref{eq:zero}), we
can rewrite the original system (\ref{model:IMHD}) as
the equivalent system
\begin{subequations}
	\begin{align}
		\boldsymbol{u}_{t}-R_{e}^{-1}\Delta\boldsymbol{u}+\nabla p+\frac{s}{R}(\nabla\boldsymbol{u})\boldsymbol{u}-\kappa\frac{s}{R}\boldsymbol{J}\times\boldsymbol{B} & =\boldsymbol{0},\label{modeleq:u}\\
		\mathrm{div}\boldsymbol{u} & =0,\label{modeleq:divu}\\
\kappa\boldsymbol{J}+\kappa\nabla \phi-\kappa\frac{s}{R}\boldsymbol{u}\times\boldsymbol{B} & =\boldsymbol{0},\label{modeleq:J}\\
		\mathrm{div}\boldsymbol{J} & =0,\label{modeleq:divJ}\\ \frac{\mathrm{d}s}{\mathrm{dt}}+sQ-\frac{1}{R}\big((\nabla\boldsymbol{u})\boldsymbol{u},\boldsymbol{u}\big) +\frac{\kappa}{R}\big(\boldsymbol{J}\times\boldsymbol{B},\boldsymbol{u}\big) +\frac{\kappa}{R}\big(\boldsymbol{u}\times\boldsymbol{B},\boldsymbol{J}\big)
& =0.\label{mdoeleq:s}
	\end{align}
	\label{model:SAV}
\end{subequations}
Besides, it is not difficult to see that the reformulated system admits
the following energy law
\begin{equation*}
	\frac{{\rm d}}{{\rm d}t}{\rm E}_{\text{IMHD-SAV}}(\boldsymbol{u})=-R_{e}^{-1}\left\Vert \nabla\boldsymbol{u}\right\Vert_{0,\mathrm{\Omega}}^{2}-\kappa\left\Vert\boldsymbol{J}\right\Vert_{0,\mathrm{\Omega}}^{2}-s^{2}Q \text{ with }
{\rm E}_{\text{IMHD-SAV}}(t)=\frac{1}{2}\left\Vert \boldsymbol{u}\right\Vert_{0,\mathrm{\Omega}}^{2}+\frac{1}{2}\left|s\right|^{2}.
\label{eq:KVSAV}
\end{equation*}
It is not difficult to observe that the equivalent system modifies the free energy with new auxiliary variable  but still maintains the energy dissipation property.
\subsection{Weak formulation}

For convenience, we define the bilinear forms $A_{1}(\cdot,\cdot)$, $A_{2}(\cdot,\cdot)$, $B(\cdot,\cdot)$, $C_{1}(\cdot,\cdot)$, $C_{2}(\cdot,\cdot)$, and $D(\cdot,\cdot)$ as
\begin{align}
\label{A1}
A_{1}(\cdot,\cdot):\boldsymbol{U}\times\boldsymbol{U}\rightarrow
\mathbb{R},\quad A_{1}(\boldsymbol{v},\boldsymbol{w})&:=\int_{\mathrm{\Omega}}\boldsymbol{v}\cdot\boldsymbol{w}\mathrm{d}\mathrm{\Omega},\\\label{A2}
A_{2}(\cdot,\cdot):\boldsymbol{E}\times\boldsymbol{E}\rightarrow
\mathbb{R},\quad A_{2}(\boldsymbol{J},\boldsymbol{K})&:=\kappa\int_{\mathrm{\Omega}}\boldsymbol{J}\cdot\boldsymbol{K}\mathrm{d}\mathrm{\Omega},\\\label{B}
B(\cdot,\cdot):\boldsymbol{U}\times\boldsymbol{U}\rightarrow
\mathbb{R},\quad B(\boldsymbol{v},\boldsymbol{w})&:=R_{e}^{-1}\int_{\mathrm{\Omega}}\nabla\boldsymbol{v}:\nabla\boldsymbol{w}\mathrm{d}\mathrm{\Omega},\\\label{C1}
C_{1}(\cdot,\cdot):\boldsymbol{U}\times Q\rightarrow
\mathbb{R},\quad C_{1}(\boldsymbol{v},q)&:=\int_{\mathrm{\Omega}}\text{div}\boldsymbol{v}q\mathrm{d}\mathrm{\Omega},\\\label{C2}
C_{2}(\cdot,\cdot):\boldsymbol{E}\times \Psi\rightarrow
\mathbb{R},\quad C_{2}(\boldsymbol{K},\psi)&:=\kappa\int_{\mathrm{\Omega}}\text{div}\boldsymbol{K}\psi\mathrm{d}\mathrm{\Omega},\\\label{D}
D(\cdot,\cdot):\boldsymbol{E}\times\boldsymbol{U}\rightarrow
\mathbb{R},\quad D(\boldsymbol{K},\boldsymbol{v})&:=\kappa\int_{\mathrm{\Omega}}\boldsymbol{K}\times\boldsymbol{B}\cdot\boldsymbol{v}\mathrm{d}\mathrm{\Omega},
\end{align}
and the trilinear form $E(\cdot;\cdot,\cdot)$ is defined by
\begin{equation}
\label{E}
E(\cdot;\cdot,\cdot):\boldsymbol{U}\times
\boldsymbol{U}\times
\boldsymbol{U}\rightarrow
\mathbb{R},\quad E(\boldsymbol{u};\boldsymbol{v},\boldsymbol{w}):=
\int_{\mathrm{\Omega}}(\nabla\boldsymbol{v})\boldsymbol{u}
\cdot\boldsymbol{w}\mathrm{d}\mathrm{\Omega}.
\end{equation}
Moreover, for a fixed $\boldsymbol{u}\in\boldsymbol{Z}$, the bilinear form $E(\cdot;,\cdot,\cdot)$ is skew-symmetric, i.e.
\begin{equation}
E(\boldsymbol{u};\boldsymbol{v},\boldsymbol{w})=-E(\boldsymbol{u};\boldsymbol{w},\boldsymbol{v})\quad\text{for all }\boldsymbol{v},\boldsymbol{w}\in\boldsymbol{U},
\end{equation}
thus we give another equivalently form
\begin{equation}
E_{s}(\boldsymbol{u};\boldsymbol{v},\boldsymbol{w}):=
\frac{1}{2}E(\boldsymbol{u};\boldsymbol{v},\boldsymbol{w})-
\frac{1}{2}E(\boldsymbol{u};\boldsymbol{w},\boldsymbol{v})\quad\text{for all }\boldsymbol{u},\boldsymbol{v},\boldsymbol{w}\in\boldsymbol{U}.
\end{equation}
Based the above notations, the variational formulation of the system (\ref{model:SAV}) reads as follows: for almost every $t\in[0,\mathrm{T}]$, we find $\big(\boldsymbol{u}(t),p(t),\boldsymbol{J}(t),\phi(t)\big)\in \boldsymbol{U}\times
Q\times\boldsymbol{E}\times
\Psi$ and $s\in\mathbb{R}$, such that for all $\big(\boldsymbol{v},q,\boldsymbol{J},\psi\big)\in \boldsymbol{U}\times
Q\times\boldsymbol{E}\times
\Psi$ there hold
\begin{subequations}
	\begin{align}
		A_{1}(\boldsymbol{u}_{t},\boldsymbol{v})+B(\boldsymbol{u},\boldsymbol{v})+C_{1}(\boldsymbol{v},p)
+\frac{s}{R}E_{s}(\boldsymbol{u};\boldsymbol{u},\boldsymbol{v}) -\frac{s}{R}D(\boldsymbol{J},\boldsymbol{v})&=0\label{variational-eq:u}\\
C_{1}(\boldsymbol{u},q) &=0,\label{variational-eq:divu}\\ A_{2}(\boldsymbol{J},\boldsymbol{K})+C_{2}(\boldsymbol{K},\phi)
+\frac{s}{R}D(\boldsymbol{K},\boldsymbol{u})&=0,\label{variational-eq:J}\\
C_{2}(\boldsymbol{J},\psi) &=0,\label{variational-eq:divJ}\\ \frac{\mathrm{d}s}{\mathrm{dt}}+sQ-\frac{1}{R}E_{s}(\boldsymbol{u};\boldsymbol{u},\boldsymbol{u}) +\frac{1}{R}D(\boldsymbol{J},\boldsymbol{u})-\frac{1}{R}D(\boldsymbol{J},\boldsymbol{u})
 & =0.\label{variational-eq:s}
	\end{align}
	\label{variational-eq}
\end{subequations}

Let us introduce the two kernel of the bilinear forms $C_{1}(\cdot,\cdot)$ and $C_{2}(\cdot,\cdot)$, i.e.
\begin{align*}
\boldsymbol{Z}:=&\{\boldsymbol{v}\in\boldsymbol{U}\quad \text{s.t.}\quad C_{1}(\boldsymbol{v},q)=0~\text{for all}~q\in Q\},\\
\boldsymbol{Y}:=&\{\boldsymbol{K}\in\boldsymbol{E}\quad \text{s.t.}\quad C_{2}(\boldsymbol{K},\psi)=0~\text{for all}~\psi\in \Psi\},
\end{align*}
then the equivalent kernel form of the system (\ref{variational-eq}) can be
inferred as follows: for almost every $t\in[0,T]$, we find $(\boldsymbol{u}(t),\boldsymbol{J}(t))\in\boldsymbol{Z}\times\boldsymbol{Y}$ and $s\in\mathbb{R}$ such that
\begin{subequations}
	\begin{align}
		A_{1}(\boldsymbol{u}_{t},\boldsymbol{v})+B(\boldsymbol{u},\boldsymbol{v})
+\frac{s}{R}E_{s}(\boldsymbol{u};\boldsymbol{u},\boldsymbol{v}) -\frac{s}{R}D(\boldsymbol{J},\boldsymbol{v})&=0,\label{variational-eq:u-kernel}\\
A_{2}(\boldsymbol{J},\boldsymbol{K})
+\frac{s}{R}D(\boldsymbol{K},\boldsymbol{u})&=0,\label{variational-eq:J-kernel}\\ \frac{\mathrm{d}s}{\mathrm{dt}}+sQ-\frac{1}{R}E_{s}(\boldsymbol{u};\boldsymbol{u},\boldsymbol{u}) +\frac{1}{R}D(\boldsymbol{J},\boldsymbol{u})-\frac{1}{R}D(\boldsymbol{J},\boldsymbol{u})
 & =0.\label{variational-eq:s-kernel}
	\end{align}
	\label{variational-eq-kernel}
\end{subequations}
for all $(\boldsymbol{v},\boldsymbol{K})\in\boldsymbol{Z}\times\boldsymbol{Y}$.

Next, we summarize some key properties of the bilinear and trilinear forms defined above, see \cite{john2016finite,zhang2021coupled,zhou2024local} for the details.
\begin{prop}
For any $\boldsymbol{u},\boldsymbol{v},\boldsymbol{w}\in\boldsymbol{U}$ and $\boldsymbol{J},\boldsymbol{K}\in\boldsymbol{E}$, the continuous forms $A_{1}(\cdot,\cdot):\boldsymbol{U}\times\boldsymbol{U}\rightarrow\mathbb{R}$,
 $A_{2}(\cdot,\cdot):\boldsymbol{E}\times\boldsymbol{E}\rightarrow\mathbb{R}$, $B(\cdot,\cdot):\boldsymbol{U}\times\boldsymbol{U}\rightarrow\mathbb{R}$ and $E_{s}(\cdot;\cdot,\cdot):\boldsymbol{U}\times\boldsymbol{U}\times\boldsymbol{U}\rightarrow\mathbb{R}$ satisfy the following properties:
\begin{align}
\label{A1-continuous}
a_{0}|\boldsymbol{u}|_{1,\mathrm{\Omega}}^2&\leq A_{1}(\boldsymbol{u},\boldsymbol{u})~\text{and}~A_{1}(\boldsymbol{u},\boldsymbol{v})\leq a_{1}|\boldsymbol{u}|_{1,\mathrm{\Omega}}|\boldsymbol{v}|_{1,\mathrm{\Omega}},\\\label{A2-continuous}
a_{2}\|\boldsymbol{K}\|_{0,\mathrm{\Omega}}^2&\leq A_{2}(\boldsymbol{K},\boldsymbol{K})~\text{and}~A_{2}(\boldsymbol{J},\boldsymbol{K})\leq a_{3}\|\boldsymbol{J}\|_{0,\mathrm{\Omega}}\|\boldsymbol{K}\|_{0,\mathrm{\Omega}},\\\label{B-continuous}
a_{4}|\boldsymbol{u}|_{1,\mathrm{\Omega}}^2&\leq B(\boldsymbol{u},\boldsymbol{u})~\text{and}~B(\boldsymbol{u},\boldsymbol{v})\leq a_{5}|\boldsymbol{u}|_{1,\mathrm{\Omega}}|\boldsymbol{v}|_{0,\mathrm{\Omega}},\\\label{C-continuous}
E_{s}(\boldsymbol{u};\boldsymbol{v},\boldsymbol{w})&\leq a_{6}|\boldsymbol{u}|_{1,\mathrm{\Omega}}|\boldsymbol{v}|_{1,\mathrm{\Omega}}|\boldsymbol{w}|_{1,\mathrm{\Omega}}
~\text{and}~E_{s}(\boldsymbol{u};\boldsymbol{v},\boldsymbol{v})=0,
\end{align}
where $a_{i}(i=0,\cdots,6)$ are positive constants and independent of $h$.
\end{prop}
\begin{prop}
The continuous forms $C_{1}(\cdot,\cdot):\boldsymbol{U}\times Q\rightarrow\mathbb{R}$ and $C_{2}(\cdot,\cdot):\boldsymbol{E}\times \Psi\rightarrow\mathbb{R}$ satisfy the following properties:
\begin{align*}
\sup_{\boldsymbol{0}\neq\boldsymbol{u}\in\boldsymbol{U}}
\frac{C_{1}(\boldsymbol{u},q)}{|\boldsymbol{u}|_{1,\mathrm{\Omega}}}\geq\gamma\|q\|_{0,\mathrm{\Omega}},\quad
\sup_{\boldsymbol{0}\neq\boldsymbol{J}\in\boldsymbol{E}}
\frac{C_{2}(\boldsymbol{J},\psi)}{\|\boldsymbol{J}\|_{0,\mathrm{\Omega}}}\geq\eta\|\psi\|_{0,\mathrm{\Omega}}\quad\text{for all }(q,\psi)\in Q\times\Psi,
\end{align*}
where $\gamma$ and $\eta$ are two positive constants only depending on $\mathrm{\Omega}$.
\end{prop}

\section{The virtual element method}
\subsection{Preliminaries and space definitions}
We decompose the domain $\mathrm{\Omega}$ into a sequence $\{\mathcal{T}_{h}\}_{h}$ consist of non-overlapping general polyhedron $K$ ($\partial K$ is the boundary of $K$, $f\in\partial K$ is a face of $K$, $\partial f$ is the boundary of $f$,  and $e$ is a edge of $K$) with $h:=\max\limits_{K\in\mathcal{T}_{h}}h_{K}$
and $h_{K}$ denotes the diameter of $K$. Let $K_{\text{v}}$, $K_{\text{e}}$, and $K_{\text{f}}$ respectively represent the number of vertices of $K$, the number of edges of $K$, and the number of faces of $K$. $|K|$ denote the measure of $K$ and $\boldsymbol{n}_{K}^{f}$ is the unit normal vector of $f$. Subsequently, we suppose each element $K$ satisfies the mesh regularity conditions for a uniform constant $\rho>0$ as follows:
\begin{itemize}
\item[\textbf{R1}]: each element $K$ is star-shaped with respect to a ball of radius $\geq\rho h_{K}$;
\item[\textbf{R2}]: each face $f$ of $K$ is star-shaped with respect to a disk of radius $\geq\rho h_{K}$;
\item[\textbf{R3}]: each edge $e$ of $K$ satisfies $\geq\rho h_{K}$.
\end{itemize}

Next, for $n\in\mathbb{N}$, we introduce some essential polynomial spaces as follows:
\begin{itemize}
\item[\textbf{S1}]: we define $\mathbb{P}_{n}(\mathrm{G})$ as a set of polynomials on $\mathrm{G}\subset\mathbb{R}^{d}$ of degree less than or equal to $n$ and respectively denote $[\mathbb{P}_{n}(\mathrm{G})]^d$ and $[\mathbb{P}_{n}(\mathrm{G})]^{d\times d}$ as vectorial and matrix polynomials ($\mathbb{P}_{-1}(\mathrm{G})=\{0\}$, $G:=f$ with $d=2$, and $G:=K$ with $d=3$);
\item[\textbf{S2}]: a set scaled monomial basis of $\mathbb{P}_{n}(f)$ is defined as
\begin{equation*}
\mathbb{M}_{n}(f):=\Big\{m_{\boldsymbol{\beta}}=\big(\frac{x-x_{K}}{h_{K}}\big)^{\beta_{1}}\big(\frac{y-y_{K}}{h_{K}}\big)^{\beta_{2}},~\text{with}~
|\boldsymbol{\beta}|\leq n,\Big\}
\end{equation*}
where the multi-index $\boldsymbol{\beta}=(\beta_{1},\beta_{2})\in\mathbb{N}^{2}$ with
$|\boldsymbol{\beta}|=\beta_{1}+\beta_{2}$ and $(x_{f},y_{f})$ denotes the centroid of $f$;
\item[\textbf{S3}]: a set scaled monomial basis of $\mathbb{P}_{n}(K)$ is defined as
\begin{equation*}
\mathbb{M}_{n}(K):=\Big\{m_{\boldsymbol{\alpha}}=\big(\frac{x-x_{K}}{h_{K}}\big)^{\alpha_{1}}\big(\frac{y-y_{K}}{h_{K}}\big)^{\alpha_{2}}\big(\frac{z-z_{K}}{h_{K}}\big)^{\alpha_{3}},~\text{with}~
|\boldsymbol{\alpha}|\leq n,\Big\}
\end{equation*}
where the multi-index $\boldsymbol{\alpha}=(\alpha_{1},\alpha_{2},\alpha_{3})\in\mathbb{N}^{3}$ with
$|\boldsymbol{\alpha}|=\alpha_{1}+\alpha_{2}+\alpha_{3}$ and $(x_{K},y_{K},z_{K})$ denotes the centroid of $K$;
\item[\textbf{S4}]: $\mathbb{\hat{P}}_{n\setminus m}(f):=\text{span}(m_{\boldsymbol{\beta}}:m+1\leq|\boldsymbol{\beta}|\leq n)$  and $\mathbb{\hat{P}}_{n\setminus m}(K):=\text{span}(m_{\boldsymbol{\alpha}}:m+1\leq|\boldsymbol{\alpha}|\leq n)$ for any $m\leq n$;
\item[\textbf{S5}]: $\mathcal{B}_{n}(\partial K):=\{v\in   C^{0}(\partial K)$ such that $v|_{f}\in\mathcal{B}_{n}(f)$     for all $f\in\partial K\}$ with $\mathcal{B}_{n}(f):=\{v\in H^{1}(f)$ such that $v$ satisfies

(a) $v|_{\partial f}\in C^{0}(\partial f):v|_{e}\in\mathbb{P}_{k}(e)$, for all $e\in\partial f\}$,

(b) $\Delta_{f}v\in\mathbb{P}_{n+1}(f)$,

(c) $\int_{f}v\hat{m}_{n+1}-\mathrm{\Pi}_{n}^{\nabla,f}v\hat{m}_{n+1}\text{d}f=0$ for all $\hat{m}_{n+1}\in\mathbb{\hat{P}}_{n\setminus m}(f)\}$.
\end{itemize}
A detailed description of the above definitions can be found in \cite{beirao2020stokes,da2016h,da2018family}.
Finally, for $n\in\mathbb{N}$, we define some projection operators as follows:
\begin{itemize}
\item[$\bigstar$]: the $\nabla$-projection $\mathrm{\Pi}_{n}^{\nabla,f}:[H^{1}(f)]^{3}\rightarrow[\mathbb{P}_{n}(f)]^{3}$, defined for all $\boldsymbol{v}\in[H^{1}(f)]^{3}$ by
\begin{equation}
\label{Pro-nabla-face}
\left\{
\begin{aligned}
\int_{f}\nabla\mathrm{\Pi}_{n}^{\nabla,f}\boldsymbol{v}:\nabla\boldsymbol{q}_{n}\mathrm{d}f &=\int_{f}\nabla\boldsymbol{v}: \nabla\boldsymbol{q}_{n}\mathrm{d}f \quad \text{for all }\boldsymbol{q}_{n}\in[\mathbb{P}_{n}(f)]^3,\\
\int_{\partial f}\mathrm{\Pi}_{n}^{\nabla,f}\boldsymbol{v}\cdot\boldsymbol{q}_{0}\mathrm{d}f  &=\int_{\partial f}\boldsymbol{v}\cdot\boldsymbol{q}_{0}\mathrm{d}f \quad \text{for all }\boldsymbol{q}_{0}\in[\mathbb{P}_{0}(f)]^3;
\end{aligned}
\right.
\end{equation}
\item[$\bigstar$]: the $L^2$-projection $\mathrm{\Pi}_{n}^{0,f}:[L^{2}(f)]^{3}\rightarrow[\mathbb{P}_{n}(f)]^{3}$, defined for all $\boldsymbol{v}\in[L^2(f)]^3$ by
\begin{equation}
\label{Pro-L2-face}
\int_{f}\mathrm{\Pi}_{n}^{0,f}\boldsymbol{v}\cdot\boldsymbol{q}_{n}\mathrm{d}f =\int_{f}\boldsymbol{v}\cdot\boldsymbol{q}_{n}\mathrm{d}f \quad\text{for all }\boldsymbol{q}_{n}\in[\mathbb{P}_{n}(f)]^3;
\end{equation}
\item[$\bigstar$]: the $\nabla$-projection $\mathrm{\Pi}_{n}^{\nabla,K}:[H^{1}(K)]^{3}\rightarrow[\mathbb{P}_{n}(K)]^{3}$, defined for all $\boldsymbol{v}\in[H^{1}(K)]^{3}$ by
\begin{equation}
\label{Pro-nabla-element}
\left\{
\begin{aligned}
\int_{K}\nabla\mathrm{\Pi}_{n}^{\nabla,K}\boldsymbol{v}:\nabla\boldsymbol{q}_{n}\mathrm{d}K &=\int_{K}\nabla\boldsymbol{v}: \nabla\boldsymbol{q}_{n}\mathrm{d}K \quad \text{for all }\boldsymbol{q}_{n}\in[\mathbb{P}_{n}(K)]^3,\\
\int_{\partial K}\mathrm{\Pi}_{n}^{\nabla,K}\boldsymbol{v}\cdot\boldsymbol{q}_{0}\mathrm{d}K  &=\int_{\partial K}\boldsymbol{v}\cdot\boldsymbol{q}_{0}\mathrm{d}K \quad \text{for all }\boldsymbol{q}_{0}\in[\mathbb{P}_{0}(K)]^3;
\end{aligned}
\right.
\end{equation}
\item[$\bigstar$]: the $L^2$-projection $\mathrm{\Pi}_{n}^{0,K}:[L^{2}(K)]^{3}\rightarrow[\mathbb{P}_{n}(K)]^{3}$, defined for all $\boldsymbol{v}\in[L^2(K)]^3$ by
\begin{equation}
\label{Pro-L2-element}
\int_{K}\mathrm{\Pi}_{n}^{0,K}\boldsymbol{v}\cdot\boldsymbol{q}_{n}\mathrm{d}K =\int_{K}\boldsymbol{v}\cdot\boldsymbol{q}_{n}\mathrm{d}K \quad\text{for all }\boldsymbol{q}_{n}\in[\mathbb{P}_{n}(K)]^3;
\end{equation}
\item[$\bigstar$]: the $\boldsymbol{L}^2$-projection $\boldsymbol{\mathrm{\Pi}}_{n}^{0,K}:[L^{2}(K)]^{3\times3}\rightarrow[\mathbb{P}_{n}(K)]^{3\times3}$, defined for all $\boldsymbol{V}\in[L^2(K)]^{3\times3}$ by
\begin{equation}
\label{Pro-grad}
\int_{K}\boldsymbol{\mathrm{\Pi}}_{n}^{0,K}\boldsymbol{V}:\boldsymbol{Q}_{n}\mathrm{d}K =\int_{K}\boldsymbol{V}:\boldsymbol{Q}_{n}\mathrm{d}K \quad\text{for all }\boldsymbol{Q}_{n}\in[\mathbb{P}_{n}(K)]^{3\times3}.
\end{equation}
\end{itemize}
\subsection{Virtual element spaces for the velocity and the pressure}
In this section, we consider the (enlarged) finite dimensional spaces for the velocity and the pressure.
Let $k_{u}\geq2$, we recall from Ref.\cite{beirao2020stokes} the local "enhanced" virtual element space:
\begin{equation*}
\boldsymbol{U}_{h}^{K}:=\Bigg\{\boldsymbol{v}_{h}\in\boldsymbol{V}_{h}^{K}~\text{s.t.}~\int_{K}\mathrm{\Pi}_{k_{u}}^{\nabla,K}\boldsymbol{v}_{h}\cdot(\boldsymbol{x}\wedge\boldsymbol{\hat{q}}_{k_{u}-1})\mathrm{d}K=\int_{K}\boldsymbol{v}_{h}\cdot(\boldsymbol{x}\wedge\boldsymbol{\hat{q}}_{k_{u}-1})\mathrm{d}K
~\text{for all }\boldsymbol{\hat{q}}_{k_{u}-1}\in[\hat{\mathbb{P}}_{k_{u}-1\setminus k_{u}-3}(K)]^{3}\Bigg\},
\end{equation*}
the enlarged virtual element space $\boldsymbol{V}_{h}^{K}$ defined by
\begin{align*}
\boldsymbol{V}_{h}^{K}:=\Bigg\{\boldsymbol{v}_{h}\in[H^{1}(K)]^3~\text{s.t}~\boldsymbol{v}_{h}|_{\partial K}\in[\mathcal{B}_{k}(\partial K)]^3,~\text{div}\boldsymbol{v}_{h}\in\mathbb{P}_{k-1}(K),~
\\ \Delta\boldsymbol{v}+\nabla p \in\boldsymbol{x}\wedge[\mathbb{P}_{k-3}(K)]^{3}~\text{for some } s\in L_{0}^{2}(K)
\Bigg\},
\end{align*}
and the main properties for a given $\boldsymbol{v}_{h}\in\boldsymbol{U}_{h}^{K}$:
\begin{itemize}
\item[\textbf{P1}]: the local degrees of freedom:
\begin{align*}
\textbf{D}^{\text{v}}_K(\boldsymbol{v}_{h})&:\text{the value of}~\boldsymbol{v}_{h}~\text{at the vertexes of}~K,\\
\textbf{D}^{\text{e}}_K(\boldsymbol{v}_{h})&:\text{the values of}~\boldsymbol{v}_{h} ~\text{at}~k_{u}-1~\text{distinct points of every edge}~ e~\text{of}~K,\\
\textbf{D}^{f}_K(\boldsymbol{v}_{h})&:\text{the face moments of }\boldsymbol{v}_{h}(\text{split into normal and tangential components})~\text{in}~f:\\
&\int_{f}(\boldsymbol{v}_{h}\cdot\boldsymbol{n}_{K}^{f})q_{k_{u}-2}\mathrm{d}f,\quad \int_{f}\boldsymbol{v}_{h,\tau}\cdot\boldsymbol{q}_{k_{u}-2}\mathrm{d}f
\quad\text{for all } q_{k_{u}-2}\in\mathbb{P}_{k_{u}-2}(f)~\text{and}~\boldsymbol{q}_{k_{u}-2}\in[\mathbb{P}_{k_{u}-2}(f)]^{2},\\
\textbf{D}^{\wedge}_{K}(\boldsymbol{v}_{h})&:\text{the volume moments of }~\boldsymbol{v}_{h}~\text{in}~K:\\&
\int_{K}\boldsymbol{v}_{h}\cdot(\boldsymbol{x}\wedge\boldsymbol{q}_{k_{u}-3})\mathrm{d}K\quad\text{for all }\boldsymbol{q}_{k_{u}-3}\in[\mathbb{P}_{k_{u}-3}(K)]^{3},\\
\textbf{D}^{\text{div}}_K(\boldsymbol{v}_{h})&:\text{the volume moments of }~\text{div}\boldsymbol{v}_{h}~\text{in}~K:\\&\int_{K}(\text{div}\boldsymbol{v}_{h})\hat{q}_{k_{u}-1}\mathrm{d}K\quad\text{for all }\hat{q}_{k_{u}-1}\in\mathbb{\hat{P}}_{k_{u}-1\setminus 0}(K),
\end{align*}
\item[\textbf{P2}]: the amount of local degrees of freedom defined in \textbf{P1} is the dimension of $\boldsymbol{U}_{h}^{K}$, which is given by
\begin{equation*}
\text{Dofs}(\boldsymbol{u})=\text{dim}(\boldsymbol{U}_{h}^{K})=3K_{v}k_{u}+3K_{e}(k_{u}-1)+3K_{f}\frac{(k_{u}-1)k_{u}}{2}+3\frac{(k_{u}-1)k_{u} (k_{u}+1)}{6},
\end{equation*}
\item[\textbf{P3}]: the polynomial space $[\mathbb{P}_{k}(K)]^{3}\subseteq\boldsymbol{U}_{h}^{K}$;
\item[\textbf{P4}]: the local degrees of freedom defined in \textbf{P1} allow us to compute exactly the following polynomial projections \cite{beirao2020stokes}:
\begin{eqnarray*}
\begin{array}{lll}
   &\mathrm{\Pi}_{k_{u}}^{\nabla,f}:[\mathcal{B}_{k_{u}}( f)]^3\rightarrow[\mathbb{P}_{k_{u}}(f)]^{3},\quad
    &\mathrm{\Pi}_{k_{u}+1}^{0,f}:[\mathcal{B}_{k_{u}}( f)]^3\rightarrow[\mathbb{P}_{k_{u}+1}(K)]^{3},\vspace{10pt}\\
    &\mathrm{\Pi}_{k_{u}}^{\nabla,K}:\boldsymbol{U}_{h}^{K}\rightarrow[\mathbb{P}_{k_{u}}(K)]^{3},\quad
    &\mathrm{\Pi}_{k_{u}}^{0,K}:\boldsymbol{U}_{h}^{K}\rightarrow[\mathbb{P}_{k_{u}}(K)]^{3},\quad
    \boldsymbol{\mathrm{\Pi}}_{k_{u}-1}^{0,K}:\nabla(\boldsymbol{U}_{h}^{K})\rightarrow[\mathbb{P}_{k_{u}-1}(K)]^{3\times3},
\end{array}
\end{eqnarray*}
\end{itemize}

Based on above local virtual spaces, we define the global virtual element space:
\begin{equation}
\label{discrete:Uh}
\boldsymbol{U}_{h}:=\{\boldsymbol{v}_{h}\in\boldsymbol{H}^{1}(\mathrm{\Omega})\quad\text{s.t.}\quad\boldsymbol{v}_{h}|_K\in\boldsymbol{U}_{h}^{K}\quad\text{for all }K\in\mathcal{T}_{h}\},
\end{equation}
and then we define a global space as
\begin{equation}
\label{discrete:Qh}
Q_{h}:=\{q\in L_{0}^{2}(\mathrm{\Omega})\quad\text{s.t.}\quad q|_{K}\in \mathbb{P}_{k_{u}-1}(K)\quad\text{for all }K\in\mathcal{T}_{h}\}.
\end{equation}
The definitions of the discrete pressure space (\ref{discrete:Qh}) and the $\boldsymbol{H}^{1}$-conforming velocity space (\ref{discrete:Uh}) imply that
\begin{equation}
\label{Div-u}
\text{div}\boldsymbol{U}_{h}\subseteq Q_{h},
\end{equation}
this key property that will lead to a divergence-free discrete velocity.
\subsection{Virtual element spaces for the current density and the electric potential}
In this section, we consider the  finite dimensional spaces for the current density and the electric potential.
Let $k_{J}\geq0$, we define the local virtual space:
\begin{equation*}
\begin{aligned}
\boldsymbol{E}_{h}^{K}:=\{\boldsymbol{v}_{h}\in\boldsymbol{H}(\textup{div};K)\cap\boldsymbol{H}(\textup{rot};K)
\;&\text{such that} \;\boldsymbol{v}_{h}\cdot\boldsymbol{n}_{K}^{f}\in\mathbb{P}_{k_{J}}(f)~ \text{for all } f~\text{of}~ K,\\&\text{div}\boldsymbol{v}_{h}\in\mathbb{P}_{k_{J}}(K) \;\text{and}\;\textbf{curl}\boldsymbol{v}_{h}\in[\mathbb{P}_{k_{J}-1}(K)]^{3}\},
\end{aligned}
\end{equation*}
The definition above is the three dimensional counterpart of the virtual elements in \cite{caceres2017mixed}, it follow in a very natural way the path of the two dimensional companions. Moreover, notice that the space $\boldsymbol{E}_{h}^{K}$ contain $[\mathbb{P}_{k_{J}}(K)]^{3}$ and this will guarantee the good approximation property of the space.
\begin{prop}
The dimension of $\boldsymbol{E}_{h}^{K}$ is given by
\begin{equation*}
\text{dim}(\boldsymbol{E}_{h}^{K})=\text{Dofs}(\boldsymbol{u})=K_{f}\frac{(k+1)(k+2)}{2}+4\frac{(k+1)(k+2)(k+3)}{6}-\frac{(k+2)(k+3)(k+4)}{6}.
\end{equation*}
Moreover, for a given $\boldsymbol{v}_{h}\in\boldsymbol{E}_{h}^{K}$, the local degrees of freedom as follows:
\begin{align*}
\textbf{D}^{\boldsymbol{n}}_K(\boldsymbol{v}_{h})&:\text{the face moments of }\boldsymbol{v}_{h}~\text{in}~f:\\
&\int_{f}(\boldsymbol{v}_{h}\cdot\boldsymbol{n}_{K}^{f})q_{k_{J}}\mathrm{d}f,
\quad\text{for all } q_{k_{J}}\in\mathbb{P}_{k_{J}}(f),\\
\textbf{D}^{\nabla}_{K}(\boldsymbol{v}_{h})&:\text{the volume moments of }~\boldsymbol{v}_{h}~\text{in}~K:\\&
\int_{K}\boldsymbol{v}_{h}\cdot\nabla \hat{q}_{k_{J}}\mathrm{d}K\quad\text{for all }\hat{q}_{k_{J}}\in\mathbb{\hat{P}}_{k_{J}\setminus0}(K),\\
\textbf{D}^{\bot}_K(\boldsymbol{v}_{h})&:\text{the volume moments of }~\boldsymbol{v}_{h}~\text{in}~K:\\&\int_{K}\boldsymbol{v}_{h}\cdot\boldsymbol{q}_{k_{J}}^{\bot}\mathrm{d}K\quad\text{for all }\boldsymbol{q}_{k_{J}}^{\bot}\in\boldsymbol{\mathcal{G}}_{k_{J}}^{\bot}(K),
\end{align*}
where $\boldsymbol{\mathcal{G}}_{k_{J}}^{\bot}(K)$
is a basis of $(\nabla\mathbb{P}_{k_{J}+1}(K))^{\bot}\cap[\mathbb{P}_{k_{J}+1}]^{3}$, which is the $[L^{2}(K)]^{3}$-orthogonal of $\nabla\mathbb{P}_{k_{J}+1}(K)$ in $[\mathbb{P}_{k_{J}+1}]^{3}$.
\end{prop}
\begin{proof}
Here, we only sketch the proof since it follows the guidelines of Sections 2.3-3.2 in \cite{da2016h} for the analogous 2D space. Then, we recall the definitions of dimension of polynomial spaces $\mathbb{P}(\mathrm{G})$ with $G:=f,K$ and the cardinality of $\boldsymbol{\mathcal{G}}_{k_{J}}^{\bot}(K)$ is $3(k+1)(k+2)(k+3)/6-(k+2)(k+3)(k+4)/6+1$ and realize the cardinalities of $\mathbb{P}_{k_{J}}(f)$ and $\mathbb{\hat{P}}_{k_{J}\setminus0}(K)$, and then by simple computations yields
\begin{align}
\text{Dofs}(\boldsymbol{u})=K_{f}\frac{(k+1)(k+2)}{2}+4\frac{(k+1)(k+2)(k+3)}{6}-\frac{(k+2)(k+3)(k+4)}{6}.
\end{align}
Moreover, it is not difficult to prove that for each $K\in\mathcal{T}_{h}$ these local degrees of freedom $\text{Dofs}(\boldsymbol{u})$ are unisolvent in $\boldsymbol{E}_{h}^{K}$.
\end{proof}
\begin{prop}
The local degrees of freedom defined in $\textbf{D}^{i}_{K}~(i:=\boldsymbol{n},\nabla,\bot)$ allow us to compute exactly the following polynomial projection:
\begin{equation*}
\mathrm{\Pi}_{k_{J}}^{0,K}:\boldsymbol{E}_{h}^{K}\rightarrow[\mathbb{P}_{k_{J}}(K)]^{3}.
\end{equation*}
\end{prop}
\begin{proof}
According to the definition of  $L^2$-projection $\mathrm{\Pi}_{n}^{0,K}$ in (\ref{Pro-L2-element}), we define
$\mathrm{\Pi}_{u_{J}}^{0,K}:\boldsymbol{E}_{h}^{K}\rightarrow[\mathbb{P}_{k_{J}}(K)]^{3}$ for all $\boldsymbol{v}_{h}\in\boldsymbol{E}_{h}^{K}$ by
\begin{equation}
\label{Current-density-L2-element}
\int_{K}\mathrm{\Pi}_{k_{J}}^{0,K}\boldsymbol{v}_{h}\cdot\boldsymbol{q}_{k_{J}}\mathrm{d}K =\int_{K}\boldsymbol{v}_{h}\cdot\boldsymbol{q}_{k_{J}}\mathrm{d}K \quad\text{for all }\boldsymbol{q}_{k_{J}}\in[\mathbb{P}_{k_{J}}(K)]^3;
\end{equation}
Obviously, we can observe that the left-hand side of (\ref{Current-density-L2-element}) is easy to calculate precisely by the $\mathrm{\Pi}_{u_{J}}^{0,K}\boldsymbol{v}_{h}\in[\mathbb{P}_{k_{J}}(K)]^3$.
Indeed, it suffices to check that the right-hand (\ref{Current-density-L2-element}) is calculable in this case. To do that, we
first mention that the degrees of freedom $\textbf{D}^{\boldsymbol{n}}_K(\boldsymbol{v}_{h})$ and $\textbf{D}^{\nabla}_{K}(\boldsymbol{v}_{h})$ imply the $\text{div}\boldsymbol{v}_{h}\in \mathbb{P}_{k_{J}}(K)$ is computable by using the identity
\begin{equation}
\label{Div-J-computable}
\int_{K}\text{div}\boldsymbol{v}_{h}q\mathrm{d}K = -\int_{K}\boldsymbol{v}_{h}\nabla q\mathrm{d}K + \int_{\partial\mathrm{\Omega}}\boldsymbol{v}_{h}\cdot\boldsymbol{n}q\mathrm{d}s \quad\text{for~all~}q\in \mathbb{P}_{k_{J}}(K).
\end{equation}
Next, given $\boldsymbol{q}\in[\mathbb{P}_{k_{J}}(K)]^3$, we know that there exist unique $\boldsymbol{q}^{\perp}\in(\nabla\mathbb{P}_{k_{J}+1}(K))^{\bot}\cap[\mathbb{P}_{k_{J}+1}]^{3}$ and $\tilde{q}\in\mathbb{P}_{k_{J}+1}(K)$, such that $\boldsymbol{q}= \boldsymbol{q}^{\perp} + \nabla\tilde{q}$. In this way, it follows that
\begin{equation*}
\int_{K}\boldsymbol{v}_{h}\boldsymbol{q}\mathrm{d}K =\int_{K} \boldsymbol{v}_{h}\cdot\boldsymbol{q}^{\perp}\mathrm{d}K
+\int_{K} \boldsymbol{v}_{h}\cdot\nabla\tilde{q}\mathrm{d}K  = \int_{K} \boldsymbol{v}_{h}\cdot\boldsymbol{q}^{\perp}\mathrm{d}K  -\int_{K} \tilde{q}\text{div}\boldsymbol{v}_{h}\mathrm{d}K +\sum_{f\in\partial K}\int_{f} \boldsymbol{v}_{h}\cdot\boldsymbol{n}_{K}^{f}\tilde{q}\mathrm{d}K,
\end{equation*}
which, according to (\ref{Div-J-computable}), $\boldsymbol{v}_{h}\cdot\boldsymbol{n}_{K}^{f}\in\mathbb{P}_{k_{J}}(f)$ and the degree of freedom $\textbf{D}^{\bot}_K(\boldsymbol{v}_{h})$, yield the
required calculation.
\end{proof}
Based on above local virtual space, we define the global virtual space:
\begin{equation}
\label{discrete:Eh}
\boldsymbol{E}_{h}:=\{\boldsymbol{v}_{h}\in\boldsymbol{H}(\text{div},\mathrm{\Omega})\quad\text{s.t.}\quad\boldsymbol{v}_{h}|_K\in\boldsymbol{E}_{h}^{K}\quad\text{for all }K\in\mathcal{T}_{h}\},
\end{equation}
and then the global pressure space is
\begin{equation}
\label{discrete:Psih}
\Psi_{h}:=\{q\in L_{0}^{2}(\mathrm{\Omega})\quad\text{s.t.}\quad q|_{K}\in \mathbb{P}_{k_{J}}(K)\quad\text{for all }K\in\mathcal{T}_{h}\}.
\end{equation}
A crucial observation is that, extending to the three dimensional case the result in \cite{beirao2016mixed,zhou2023full}, the proposed discrete electric potential space (\ref{discrete:Psih}) and the $\boldsymbol{H}(\text{div})$-conforming current density space (\ref{discrete:Eh}) are imply that
\begin{equation}
\label{Div-J}
\text{div}\boldsymbol{E}_{h}\subseteq \Psi_{h},
\end{equation}
this key property that will lead to a divergence-free discrete current density.
\subsection{Discrete forms}
In this section, we will define some computable discrete local bilinear and trilinear forms in the standard VEM framework.

Firstly, for any $\boldsymbol{u},\boldsymbol{v},\boldsymbol{w}\in\boldsymbol{U}_{h}$, $\boldsymbol{J},\boldsymbol{K}\in\boldsymbol{E}_{h}$, $q\in Q_{h}$ and $\psi\in \Psi_{h}$, we decompose the forms defined in (\ref{A1})-(\ref{D}) into local contributions as follows:
\begin{align*}
A_{1}(\boldsymbol{u},\boldsymbol{v}):=\sum\limits_{K\in\mathcal{T}_{h}}A_{1}^{K}(\boldsymbol{u},\boldsymbol{v}),\quad
A_{2}(\boldsymbol{J},\boldsymbol{K}):=\sum\limits_{K\in\mathcal{T}_{h}}A_{2}^{K}(\boldsymbol{J},\boldsymbol{K}),\\
B(\boldsymbol{u},\boldsymbol{v}):=\sum\limits_{K\in\mathcal{T}_{h}}B^{K}(\boldsymbol{u},\boldsymbol{v}),\quad
D(\boldsymbol{K},\boldsymbol{v}):=\sum\limits_{K\in\mathcal{T}_{h}}D^{K}(\boldsymbol{K},\boldsymbol{v}),\\
C_{1}(\boldsymbol{u},q):=\sum\limits_{K\in\mathcal{T}_{h}}C_{1}^{K}(\boldsymbol{u},q),\quad
C_{2}(\boldsymbol{K},\psi):=\sum\limits_{K\in\mathcal{T}_{h}}C_{2}^{K}(\boldsymbol{K},\psi),\quad
E_{s}(\boldsymbol{u};\boldsymbol{v},\boldsymbol{w}):=\sum\limits_{K\in\mathcal{T}_{h}}D_{s}^{E}(\boldsymbol{u};\boldsymbol{v},\boldsymbol{w}),
\end{align*}

Secondly, to approximate $A_{1}^{K}(\cdot,\cdot):\boldsymbol{U}_{h}^{K}\times\boldsymbol{U}_{h}^{K}\rightarrow\mathbb{R}$,  $A_{2}^{K}(\cdot,\cdot):\boldsymbol{E}_{h}^{K}\times\boldsymbol{E}_{h}^{K}\rightarrow\mathbb{R}$, and $B^{K}(\cdot,\cdot):\boldsymbol{U}_{h}^{K}\times\boldsymbol{U}_{h}^{K}\rightarrow\mathbb{R}$,  we define computable discrete local bilinear forms
\begin{align}
\label{A1_h^K}
A_{1,h}^{K}(\boldsymbol{u}_{h},\boldsymbol{v}_{h})&:=A_{1}^{K}(\mathrm{\Pi}_{k_{u}}^{0,K}\boldsymbol{u}_{h},\mathrm{\Pi}_{k_{u}}^{0,K}\boldsymbol{v}_{h})+\mathcal{S}_{1}^{K}(\boldsymbol{u}_{h}-\mathrm{\Pi}_{k_{u}}^{0,K}\boldsymbol{u}_{h},\boldsymbol{v}_{h}-\mathrm{\Pi}_{k_{u}}^{0,K}\boldsymbol{v}_{h}),\\\label{A2_h^K}
A_{2,h}^{K}(\boldsymbol{J}_{h},\boldsymbol{K}_{h})&:=A_{2}^{K}(\mathrm{\Pi}_{k_{J}}^{0,K}\boldsymbol{J}_{h},\mathrm{\Pi}_{k_{J}}^{0,K}\boldsymbol{K}_{h})+\mathcal{S}_{2}^{K}(\boldsymbol{J}_{h}-\mathrm{\Pi}_{k_{J}}^{0,K}\boldsymbol{J}_{h},\boldsymbol{K}_{h}-\mathrm{\Pi}_{k_{J}}^{0,K}\boldsymbol{K}_{h}),\\\label{B_h^K}
B_{h}^{K}(\boldsymbol{u}_{h},\boldsymbol{v}_{h})&:=B^{K}(\mathrm{\Pi}_{k_{u}}^{\nabla,K}\boldsymbol{u}_{h},\mathrm{\Pi}_{k_{u}}^{\nabla,K}\boldsymbol{v}_{h})+\mathcal{S}_{3}^{K}(\boldsymbol{u}_{h}-\mathrm{\Pi}_{k_{u}}^{\nabla,K}\boldsymbol{u}_{h},\boldsymbol{v}_{h}-\mathrm{\Pi}_{k_{u}}^{\nabla,K}\boldsymbol{v}_{h}),
\end{align}
for all $\boldsymbol{u}_{h},\boldsymbol{v}_{h}\in\boldsymbol{U}_{h}^{K}$ and $\boldsymbol{J}_{h},\boldsymbol{K}_{h}\in\boldsymbol{E}_{h}^{K}$, where the symmetric, positive definite and computable discrete bilinear forms
 $\mathcal{S}_{1}^{K}:\boldsymbol{U}_{h}^{K}\times\boldsymbol{U}_{h}^{K}\rightarrow\mathbb{R}$, $\mathcal{S}_{2}^{K}:\boldsymbol{E}_{h}^{K}\times\boldsymbol{E}_{h}^{K}\rightarrow\mathbb{R}$, and $\mathcal{S}_{3}^{K}:\boldsymbol{U}_{h}^{K}\times\boldsymbol{U}_{h}^{K}\rightarrow\mathbb{R}$ satisfy
\begin{align}
\label{S_1^K}
\alpha_{0}A_{1}^{K}(\boldsymbol{v}_{h},\boldsymbol{v}_{h})\leq\mathcal{S}_{1}^{K}(\boldsymbol{v}_{h},\boldsymbol{v}_{h})&\leq \alpha_{1}A_{1}^{K}(\boldsymbol{v}_{h},\boldsymbol{v}_{h}),\\\label{S_2^K}
\alpha_{2}A_{2}^{K}(\boldsymbol{K}_{h},\boldsymbol{K}_{h})\leq\mathcal{S}_{2}^{K}(\boldsymbol{K}_{h},\boldsymbol{K}_{h})&\leq \alpha_{3}A_{2}^{K}(\boldsymbol{K}_{h},\boldsymbol{K}_{h}),\\\label{S_3^K}
\beta_{0}B^{K}(\boldsymbol{v}_{h},\boldsymbol{v}_{h})\leq\mathcal{S}_{3}^{K}(\boldsymbol{v}_{h},\boldsymbol{v}_{h})&\leq
\beta_{1}B^{K}(\boldsymbol{v}_{h},\boldsymbol{v}_{h}),
\end{align}
where $\alpha_{i}~(i=0,1,2,3)$ and $\beta_{j}~(j=0,1)$ are positive constants independent of $K$.
\begin{rem}
In the VEM literature \cite{brenner2018virtual}, there are more choices for VEM stability terms, and in this paper we consider simple and effective choices as follows
\begin{equation*}
\mathcal{S}_{1}^{K}(\cdot,\cdot)=|K|\sum_{i=1}^{\text{Dofs}(\boldsymbol{u})}\mathcal{X}_{i}(\cdot)\mathcal{X}_{i}(\cdot),
\quad\mathcal{S}_{2}^{K}(\cdot,\cdot)=|K|\sum_{i=1}^{\text{Dofs}(\boldsymbol{J})}\mathcal{X}_{i}(\cdot)\mathcal{X}_{i}(\cdot),\quad
\mathcal{S}_{3}^{K}(\cdot,\cdot)=h_{K}\sum_{i=1}^{\text{Dofs}(\boldsymbol{u})}\mathcal{X}_{i}(\cdot)\mathcal{X}_{i}(\cdot),
\end{equation*}
where $\mathcal{X}_{i}$ is the operator that selects the $i$-th degrees of freedom.
\end{rem}
Furthermore, for all  $(\boldsymbol{v}_{h},\boldsymbol{K}_{h})\in\boldsymbol{U}_{h}^{K}\times\boldsymbol{E}_{h}^{K}$ and $(\boldsymbol{q}_{k_{u}},\boldsymbol{q}_{k_{J}})\in[\mathbb{P}_{k_{u}}]^{3}\times[\mathbb{P}_{k_{J}}]^{3}$, we note that the following some important properties.
\begin{itemize}
\item[$\bullet$] consistencies: the definitions of (\ref{Pro-nabla-element})-(\ref{Pro-L2-element}) imply
\begin{equation}
\label{consistency}
A_{1,h}^{K}(\boldsymbol{v}_{h},\boldsymbol{q}_{k_{u}}) = A_{1}^{K}(\boldsymbol{v}_{h},\boldsymbol{q}_{k_{u}}),\quad
A_{2,h}^{K}(\boldsymbol{K}_{h},\boldsymbol{q}_{k_{J}}) = A_{2}^{K}(\boldsymbol{K}_{h},\boldsymbol{q}_{k_{J}}),\quad
B_{h}^{K}(\boldsymbol{v}_{h},\boldsymbol{q}_{k_{u}}) = B^{K}(\boldsymbol{v}_{h},\boldsymbol{q}_{k_{u}}),
\end{equation}
\item[$\bullet$] stabilities: the properties (\ref{S_1^K})-(\ref{S_3^K}) imply
\begin{align}
\label{StableA1}
\alpha_{0}A_{1}^{K}(\boldsymbol{v}_{h},\boldsymbol{v}_{h})\leq A_{1,h}^{K}(\boldsymbol{v}_{h},\boldsymbol{v}_{h})&\leq \alpha_{1}A_{1}^{K}(\boldsymbol{v}_{h},\boldsymbol{v}_{h}),\\\label{StableA2}
\alpha_{2}A_{2}^{K}(\boldsymbol{K}_{h},\boldsymbol{K}_{h})\leq A_{2,h}^{K}(\boldsymbol{K}_{h},\boldsymbol{K}_{h})&\leq \alpha_{3}A_{2}^{K}(\boldsymbol{K}_{h},\boldsymbol{K}_{h}),\\\label{StableB}
\beta_{0}B^{K}(\boldsymbol{v}_{h},\boldsymbol{v}_{h})\leq B_{h}^{K}(\boldsymbol{v}_{h},\boldsymbol{v}_{h})&\leq
\beta_{1}B^{K}(\boldsymbol{v}_{h},\boldsymbol{v}_{h}).
\end{align}
\end{itemize}
Obviously, we can clearly sum up the local contributions to define the global approximated bilinear form $A_{1,h}(\cdot,\cdot):\boldsymbol{U}_{h}\times\boldsymbol{U}_{h}\rightarrow\mathbb{R}$, $A_{2,h}(\cdot,\cdot):\boldsymbol{E}_{h}\times\boldsymbol{E}_{h}\rightarrow\mathbb{R}$, and $B_{h}(\cdot,\cdot):\boldsymbol{U}_{h}\times\boldsymbol{U}_{h}\rightarrow\mathbb{R}$ as follows:
\begin{equation*}
A_{1,h}(\boldsymbol{u}_{h},\boldsymbol{v}_{h})=\sum_{K\in\mathcal{T}_{h}}A_{1,h}^{K}(\boldsymbol{u}_{h},\boldsymbol{v}_{h}),\quad
A_{2,h}(\boldsymbol{J}_{h},\boldsymbol{K}_{h})=\sum_{K\in\mathcal{T}_{h}}A_{2,h}^{K}(\boldsymbol{J}_{h},\boldsymbol{K}_{h}),\quad B_{h}(\boldsymbol{u}_{h},\boldsymbol{v}_{h})=\sum_{K\in\mathcal{T}_{h}}B_{h}^{K}(\boldsymbol{u}_{h},\boldsymbol{v}_{h}),
\end{equation*}
for all $\boldsymbol{u}_{h},\boldsymbol{v}_{h}\in\boldsymbol{U}_{h}$ and $\boldsymbol{J}_{h},\boldsymbol{K}_{h}\in\boldsymbol{E}_{h}$.
These inner product induce the norms
\begin{equation}
\label{define-norm}
\|\boldsymbol{v}_{h}\|_{A_{1}} = [A_{1,h}(\boldsymbol{v}_{h},\boldsymbol{v}_{h})]^{\frac{1}{2}},\quad
\|\boldsymbol{K}_{h}\|_{A_{2}} =[A_{2,h}(\boldsymbol{K}_{h},\boldsymbol{K}_{h})]^{\frac{1}{2}},\quad
\|\boldsymbol{v}_{h}\|_{B} =[B_{h}(\boldsymbol{v}_{h},\boldsymbol{v}_{h})]^{\frac{1}{2}}.
\end{equation}

Thirdly, the global bilinear form $C_{1}(\cdot,\cdot):\boldsymbol{U}_{h}\times Q_{h}\rightarrow\mathbb{R}$ and $C_{1}(\cdot,\cdot):\boldsymbol{E}_{h}\times \Psi_{h}\rightarrow\mathbb{R}$  can be directly defined by
\begin{align}
\label{global C1}
C_{1}(\boldsymbol{v}_{h},q) &= \sum_{K\in\mathcal{T}_{h}}C_{1}^{K}(\boldsymbol{v}_{h},q)=\sum_{K\in\mathcal{T}_{h}}\int_{K}\text{div}\boldsymbol{v}_{h}q\mathrm{d}K\quad \text{for all }(\boldsymbol{v}_{h},q)\in \boldsymbol{U}_{h}\times Q_{h},\\\label{global_C2}
C_{2}(\boldsymbol{K}_{h},\psi) &= \sum_{K\in\mathcal{T}_{h}}C_{2}^{K}(\boldsymbol{K}_{h},q)=\sum_{K\in\mathcal{T}_{h}}\int_{K}\text{div}\boldsymbol{K}_{h}\psi\mathrm{d}K\quad \text{for all }(\boldsymbol{K}_{h},\psi)\in \boldsymbol{E}_{h}\times \Psi_{h},
\end{align}
without introducing any approximation. The reason are that (\ref{global C1}) is computable from $\textbf{D}^{i}_{K}~(i:=\text{n},\text{e},\text{div})$ and (\ref{global_C2}) is also computable from $\textbf{D}^{j}_{K}~(j:=\boldsymbol{n},\nabla)$.

Fourthly, to approximate $E_{s}^{K}(\cdot;\cdot,\cdot):\boldsymbol{U}_{h}^{K}\times\boldsymbol{U}_{h}^{K}\times\boldsymbol{U}_{h}^{K}\rightarrow\mathbb{R}$, we define a computable discrete local trilinear form
\begin{align}
\label{E_h^K}
E_{s,h}^{K}(\boldsymbol{u}_{h};\boldsymbol{v}_{h},\boldsymbol{w}_{h}):=\frac{1}{2}\int_{K}\big[(\boldsymbol{\mathrm{\Pi}}_{k_{u}-1}^{0}\nabla\boldsymbol{v}_{h})(\mathrm{\Pi}_{k_{u}}^{0}\boldsymbol{u}_{h})\big]\cdot\mathrm{\Pi}_{k_{u}}^{0}\boldsymbol{w}_{h}-\big[(\boldsymbol{\mathrm{\Pi}}_{k_{u}-1}^{0}\nabla\boldsymbol{w}_{h})(\mathrm{\Pi}_{k_{u}}^{0}\boldsymbol{u}_{h})\big]\cdot\mathrm{\Pi}_{k_{u}}^{0}\boldsymbol{v}_{h}\mathrm{d}K,
\end{align}
for all $\boldsymbol{u}_{h},\boldsymbol{v}_{h},\boldsymbol{w}_{h}\in\boldsymbol{U}_{h}^{K}$.
As usual, we can simply sum up the local contributions to define the global approximated trilinear form $E_{s,h}(\cdot;\cdot,\cdot):\boldsymbol{U}_{h}\times\boldsymbol{U}_{h}\times\boldsymbol{U}_{h}\rightarrow\mathbb{R}$ as follows:
\begin{equation}
\label{global D}
E_{s,h}(\boldsymbol{u}_{h};\boldsymbol{v}_{h},\boldsymbol{w}_{h}):=\sum_{K\in\mathcal{T}_{h}}E_{s,h}^{K}(\boldsymbol{u}_{h};\boldsymbol{v}_{h},\boldsymbol{w}_{h}),\quad
\text{for all } \boldsymbol{u}_{h},\boldsymbol{v}_{h},\boldsymbol{w}_{h}\in\boldsymbol{U}_{h}.
\end{equation}
Finally, to approximate $D^{K}(\cdot,\cdot):\boldsymbol{E}_{h}^{K}\times\boldsymbol{U}_{h}^{K}\rightarrow\mathbb{R}$, we define a computable discrete local bilinear form
\begin{align}
\label{D_h^K}
D_{h}^{K}(\boldsymbol{K}_{h},\boldsymbol{v}_{h}):=\kappa\int_{K}(\mathrm{\Pi}_{k_{J}}^{0,K}\boldsymbol{K}_{h}\times\boldsymbol{B})\cdot\mathrm{\Pi}_{k_{u}}^{0,K}\boldsymbol{v}_{h}\mathrm{d}K\quad\text{for all }(\boldsymbol{K}_{h},\boldsymbol{v}_{h})\in\boldsymbol{E}_{h}^{K}\times\boldsymbol{U}_{h}^{K},
\end{align}
and the global approximated bilinear form $D_{h}(\cdot,\cdot):\boldsymbol{E}_{h}\times\boldsymbol{U}_{h}\rightarrow\mathbb{R}$ is defined as follows:
\begin{equation*}
D_{h}(\boldsymbol{K}_{h},\boldsymbol{v}_{h}):=\sum_{K\in\mathcal{T}_{h}}D_{h}^{K}(\boldsymbol{K}_{h},\boldsymbol{v}_{h})\quad
\text{for all } (\boldsymbol{K}_{h},\boldsymbol{v}_{h})\in\boldsymbol{E}_{h}\times\boldsymbol{U}_{h}.
\end{equation*}

Next, we summarize some of their properties, see \cite{beirao2020stokes,zhou2023full} for the details.
\begin{prop}
For any $\boldsymbol{u_{h}},\boldsymbol{v_{h}},\boldsymbol{w_{h}}\in\boldsymbol{U}_{h}$ and $\boldsymbol{J_{h}},\boldsymbol{K_{h}}\in\boldsymbol{E}_{h}$, the global forms $A_{1,h}(\cdot,\cdot):\boldsymbol{U}_{h}\times\boldsymbol{U}_{h}\rightarrow\mathbb{R}$,
$A_{2,h}(\cdot,\cdot):\boldsymbol{E}_{h}\times\boldsymbol{E}_{h}\rightarrow\mathbb{R}$,
$B_{h}(\cdot,\cdot):\boldsymbol{U}_{h}\times\boldsymbol{U}_{h}\rightarrow\mathbb{R}$ and $E_{s,h}(\cdot;\cdot,\cdot):\boldsymbol{U}_{h}\times\boldsymbol{U}_{h}\times\boldsymbol{U}_{h}\rightarrow\mathbb{R}$ satisfy the following properties:
\begin{align}
\label{A1_h-disrete}
\alpha_{0}|\boldsymbol{u}_{h}|_{1,\mathrm{\Omega}}^2&\leq A_{1,h}(\boldsymbol{u}_{h},\boldsymbol{u}_{h})~\text{and}~A_{1,h}(\boldsymbol{u}_{h},\boldsymbol{v}_{h})\leq \alpha_{1}|\boldsymbol{u}_{h}|_{1,\mathrm{\Omega}}|\boldsymbol{v}_{h}|_{1,\mathrm{\Omega}},\\\label{A2_h-disrete}
\alpha_{2}\|\boldsymbol{K}_{h}\|_{0,\mathrm{\Omega}}^2&\leq A_{2,h}(\boldsymbol{K}_{h},\boldsymbol{K}_{h})~\text{and}~A_{2,h}(\boldsymbol{J}_{h},\boldsymbol{K}_{h})\leq \alpha_{3}\|\boldsymbol{J}_{h}\|_{0,\mathrm{\Omega}}\|\boldsymbol{K}_{h}\|_{0,\mathrm{\Omega}},\\\label{B_h-disrete}
\alpha_{4}|\boldsymbol{u}_{h}|_{0,\mathrm{\Omega}}^2&\leq B_{h}(\boldsymbol{u}_{h},\boldsymbol{u}_{h})~\text{and}~B_{h}(\boldsymbol{u}_{h},\boldsymbol{v}_{h})\leq \alpha_{5}|\boldsymbol{u}_{h}|_{1,\mathrm{\Omega}}|\boldsymbol{v}_{h}|_{0,\mathrm{\Omega}},\\\label{C_h-disrete}
E_{s,h}(\boldsymbol{u}_{h};\boldsymbol{v}_{h},\boldsymbol{w}_{h})&\leq \alpha_{6}|\boldsymbol{u}_{h}|_{1,\mathrm{\Omega}}|\boldsymbol{v}_{h}|_{1,\mathrm{\Omega}}|\boldsymbol{w}_{h}|_{1,\mathrm{\Omega}}
~\text{and}~E_{s,h}(\boldsymbol{u}_{h};\boldsymbol{v}_{h},\boldsymbol{v}_{h})=0,
\end{align}
where $\alpha_{i}(i=0,\cdots,6)$ are positive constants and independent of $h$.
\end{prop}
\begin{prop}
\label{discrete-inf-sup}
The discrete form $C_{1}(\cdot,\cdot):\boldsymbol{U}_{h}\times Q_{h}\rightarrow\mathbb{R}$ and $C_{2}(\cdot,\cdot):\boldsymbol{E}_{h}\times \Psi_{h}\rightarrow\mathbb{R}$ satisfy the following properties:
\begin{equation}
\sup_{\boldsymbol{0}\neq\boldsymbol{u}_{h}\in\boldsymbol{U}_{h}}
\frac{C_{1}(\boldsymbol{u}_{h},q_{h})}{|\boldsymbol{u}_{h}|_{1,\mathrm{\Omega}}}\geq\gamma^{*}\|q_{h}\|_{0,\mathrm{\Omega}},\quad\sup_{\boldsymbol{0}\neq\boldsymbol{K}_{h}\in\boldsymbol{E}_{h}}
\frac{C_{2}(\boldsymbol{K}_{h},\psi_{h})}{\|\boldsymbol{K}_{h}\|_{0,\mathrm{\Omega}}}\geq\eta^{*}\|\psi_{h}\|_{0,\mathrm{\Omega}},\quad\text{for all }(q,\psi_{h})\in Q_{h}\times\Psi_{h},
\end{equation}
where $\gamma^{*}>0$ and $\eta^{*}>0$ are two constants only depending on $\mathrm{\Omega}$.
\end{prop}
\section{Numerical discrete schemes}
At the beginning of this section, we introduce a sequence of time steps $t^{n}=n\Delta t~(\text{where}~n=0,1,2,...,N)$ with $\Delta t = T/N$  being the time step. The unknown fields $(\boldsymbol{u}_{h},p_{h},\boldsymbol{J}_{h},\phi_{h})$ at these discrete times $t^n$ are approximated by the discrete quantities
\begin{equation*}
\boldsymbol{u}_{h}^{n}\approx\boldsymbol{u}(t^{n}),\quad p_{h}^{n}\approx p(t^{n}),\quad\boldsymbol{J}_{h}^{n}\approx\boldsymbol{J}(t^{n}),\quad \phi_{h}^{n}\approx \phi(t^{n}),\quad s^{n}\approx s(t^{n}),\quad R^{n}= R(t^{n}).
\end{equation*}
To facilitate writing, we also define the discrete differential operators $\delta_{t}$ and $\delta_{t}^{2}$ as:
\begin{equation*}
\delta_{t}^{n+1}\boldsymbol{u}_{h} = \frac{\boldsymbol{u}_{h}^{n+1}-\boldsymbol{u}_{h}^{n}}{\Delta t},\quad\delta_{t}^{n+1}s= \frac{s^{n+1}-s^{n}}{\Delta t},\quad
\delta_{t}^{2,n+1}\boldsymbol{u}_{h} = \frac{3\boldsymbol{u}_{h}^{n+1}-4\boldsymbol{u}_{h}^{n}+\boldsymbol{u}_{h}^{n-1}}{2\Delta t},\quad \boldsymbol{\tilde{u}}_{h} = 2\boldsymbol{u}_{h}^{n}-\boldsymbol{u}^{n-1}.
\end{equation*}
\subsection{A fully discrete first-order SAV scheme \label{section:first-order}}
In this section, we now construct a fully discrete first-order SAV virtual element approximation of equations (\ref{variational-eq}). The initial data $(\boldsymbol{u}_{h}^{0},s^{0})\in\boldsymbol{U}_{h}\times\mathbb{R}$ are given by $(\boldsymbol{u}_{0},s(0))$, and the other initial data $\boldsymbol{J}_{h}^{0}\in
\boldsymbol{E}_{h}$ are obtained by solving the following equations:
\begin{equation}
\label{Initial-first}
\begin{aligned}
A_{2,h}(\boldsymbol{J}_{h}^{0},\boldsymbol{K}_{h})+C_{2,h}(\boldsymbol{K}_{h},\phi_{h}^{0})
+\frac{s^{0}}{R^{0}}D_{h}(\boldsymbol{K}_{h},\boldsymbol{u}_{h}^{0})&=0,\\
C_{2,h}(\boldsymbol{J}_{h}^{0},\psi_{h}) &=0,
\end{aligned}
\end{equation}
for all $(\boldsymbol{K}_{h},\psi_{h})\in
\boldsymbol{E}_{h}\times\Psi_{h}$. Assuming we have computed $(\boldsymbol{u}_{h}^{n},p_{h}^{n},\boldsymbol{J}_{h}^{n},\phi_{h}^{n},s^{n})\in \boldsymbol{U}_{h}\times
Q_{h}\times\boldsymbol{E}_{h}\times\Psi_{h}\times\mathbb{R}$ for $n>0$, and then we apply the backward Euler discretization in time and the virtual element method in spaces to deduce the numerical scheme reads as: find $\big(\boldsymbol{u}_{h}^{n+1},p_{h}^{n+1},\boldsymbol{J}_{h}^{n+1},\phi_{h}^{n+1},s^{n+1}\big)\in \boldsymbol{U}_{h}\times
Q_{h}\times\boldsymbol{E}_{h}\times\Psi_{h}\times\mathbb{R}$ such that
\begin{subequations}
	\label{eq:full-discrete-first-order}
	\begin{align}
\label{eq:full-discrete-first-order-u}
A_{1,h}(\delta_{t}\boldsymbol{u}_{h}^{n+1},\boldsymbol{v}_{h})+B_{h}(\boldsymbol{u}_{h}^{n+1},\boldsymbol{v}_{h})+C_{1,h}(\boldsymbol{v}_{h},p_{h}^{n+1})
+\frac{s^{n+1}}{R^{n+1}}E_{s,h}(\boldsymbol{u}_{h}^{n};\boldsymbol{u}_{h}^{n},\boldsymbol{v}_{h})-\frac{s^{n+1}}{R^{n+1}}D_{h}(\boldsymbol{J}_{h}^{n},\boldsymbol{v}_{h})&=0,\\\label{eq:full-discrete-first-order-p}
C_{1,h}(\boldsymbol{u}_{h}^{n+1},q_{h}) &=0,\\\label{eq:full-discrete-first-order-J}
A_{2,h}(\boldsymbol{J}_{h}^{n+1},\boldsymbol{K}_{h})+C_{2,h}(\boldsymbol{K}_{h},\phi_{h}^{n+1})
+\frac{s^{n+1}}{R^{n+1}}D_{h}(\boldsymbol{K}_{h},\boldsymbol{u}_{h}^{n})&=0,\\\label{eq:full-discrete-first-order-phi}
C_{2,h}(\boldsymbol{J}_{h}^{n+1},\psi_{h}) &=0,\\ \label{eq:full-discrete-first-order-s} \delta_{t}s^{n+1}+s^{n+1}Q-\frac{1}{R^{n+1}}E_{s,h}(\boldsymbol{u}_{h}^{n};\boldsymbol{u}_{h}^{n},\boldsymbol{u}_{h}^{n+1})+\frac{1}{R^{n+1}}D_{h}(\boldsymbol{J}_{h}^{n},\boldsymbol{u}_{h}^{n+1})
-\frac{1}{R^{n+1}}D_{h}(\boldsymbol{J}_{h}^{n+1},\boldsymbol{u}_{h}^{n})& =0,
	\end{align}
\end{subequations}
for all $\big(\boldsymbol{v}_{h},q_{h},\boldsymbol{K}_{h},\psi_{h}\big)\in \boldsymbol{U}_{h}\times
Q_{h}\times\boldsymbol{E}_{h}\times
\Psi_{h}$  there hold. Subsequently, we establish the unconditionally energy stability of the formulation (\ref{eq:full-discrete-first-order}) as follows.
\begin{thm}
The formulation (\ref{eq:full-discrete-first-order}) is unconditionally energy stable in the sense that the following energy estimate
\begin{equation}
\label{Energy-first-order}
\delta_{t}^{n+1}\mathrm{E}^{n+1}\leq-\|\boldsymbol{u}_{h}^{n+1}\|_{B}^{2}-\|\boldsymbol{J}_{h}^{n+1}\|_{A_{2}}^{2}-Q|s^{n+1}|^2\text{ with } \mathrm{E}^{n+1} = \frac{1}{2}\|\boldsymbol{u}_{h}^{n+1}\|_{A_{1}}^{2}+ \frac{1}{2}|s^{n+1}|^2.
\end{equation}

\end{thm}

\begin{proof}
Taking $\big(\boldsymbol{v}_{h},q_{h},\boldsymbol{K}_{h},\psi_{h}\big)=\big(\boldsymbol{u}_{h}^{n+1},p_{h}^{n+1},\boldsymbol{J}_{h}^{n+1},\phi_{h}^{n+1}\big)$ in equations (\ref{eq:full-discrete-first-order-u})-(\ref{eq:full-discrete-first-order-phi}) and applying the identity $2a(a-b)=a^2-b^2+(a-b)^2$ and the definitions of norm in (\ref{define-norm}) lead to
\begin{equation}
\label{Energy-first-order-first}
	\begin{aligned}
\frac{1}{2\Delta t}\big(\|\boldsymbol{u}_{h}^{n+1}\|_{A_{1}}^{2}-\|\boldsymbol{u}_{h}^{n}\|_{A_{1}}^{2}
+\|\boldsymbol{u}_{h}^{n+1}-\boldsymbol{u}_{h}^{n}\|_{A_{1}}^{2}\big)
+\|\boldsymbol{u}_{h}^{n+1}\|_{B}^{2}+\|\boldsymbol{J}_{h}^{n+1}\|_{A_{2}}^{2}\\
+\frac{s^{n+1}}{R^{n+1}}E_{s,h}(\boldsymbol{u}_{h}^{n};\boldsymbol{u}_{h}^{n},\boldsymbol{u}_{h}^{n+1})
-\frac{s^{n+1}}{R^{n+1}}D_{h}(\boldsymbol{J}_{h}^{n},\boldsymbol{u}_{h}^{n+1})
+\frac{s^{n+1}}{R^{n+1}}D_{h}(\boldsymbol{J}_{h}^{n+1},\boldsymbol{u}_{h}^{n})&=0.\\
	\end{aligned}
\end{equation}
Multiplying (\ref{eq:full-discrete-first-order-s}) by $s^{n+1}$ leads to
\begin{equation}
\label{Energy-first-order-second}
	\begin{aligned}
\frac{1}{2\Delta t}\big(|s^{n+1}|^2-|s^{n}|^2+|s^{n+1}-s^{n}|^2\big)+Q|s^{n+1}|^2&\\
-\frac{s^{n+1}}{R^{n+1}}E_{s,h}(\boldsymbol{u}_{h}^{n};\boldsymbol{u}_{h}^{n},\boldsymbol{u}_{h}^{n+1})
+\frac{s^{n+1}}{R^{n+1}}D_{h}(\boldsymbol{J}_{h}^{n},\boldsymbol{u}_{h}^{n+1})
-\frac{s^{n+1}}{R^{n+1}}D_{h}(\boldsymbol{J}_{h}^{n+1},\boldsymbol{u}_{h}^{n})&=0.
	\end{aligned}
\end{equation}
Combining (\ref{Energy-first-order-first}) with (\ref{Energy-first-order-second})  results in obtaining
\begin{equation*}
	\begin{aligned}
\frac{1}{2\Delta t}\big(\|\boldsymbol{u}_{h}^{n+1}\|_{A_{1}}^{2}-\|\boldsymbol{u}_{h}^{n}\|_{A_{1}}^{2}
+\|\boldsymbol{u}_{h}^{n+1}-\boldsymbol{u}_{h}^{n}\|_{A_{1}}^{2}\big)
+\|\boldsymbol{u}_{h}^{n+1}\|_{B}^{2}+\|\boldsymbol{J}_{h}^{n+1}\|_{A_{2}}^{2}&\\
+\frac{1}{2\Delta t}\big(|s^{n+1}|^2-|s^{n}|^2+|s^{n+1}-s^{n}|^2\big)
+Q|s^{n+1}|^2&=0.
	\end{aligned}
\end{equation*}
This yields the conclusion (\ref{Energy-first-order}).
\end{proof}

\begin{prop}
\label{Prop:first-divergence free}
Let $\big(\boldsymbol{u}_{h}^{n+1},p_{h}^{n+1},\boldsymbol{J}_{h}^{n+1},\phi_{h}^{n+1},s^{n+1}\big)\in \boldsymbol{U}_{h}\times
Q_{h}\times\boldsymbol{E}_{h}\times\Psi_{h}\times\mathbb{R}$ be the solution to the equations (\ref{eq:full-discrete-first-order}), the scheme fully satisfies the properties of mass conservation and charge conservation. Namely, $\text{div}\boldsymbol{u}_{h}^{n+1}=0$ and $\text{div}\boldsymbol{J}_{h}^{n+1}=0$.
\end{prop}
\begin{proof}
It is worth noting that for any $(q_{h},\psi_{h})\in Q_{h}\times\Psi_{h}$, there hold
\begin{equation*}
C_{1,h}(\boldsymbol{u}_{h}^{n+1},q_{h}) =0,\quad C_{2,h}(\boldsymbol{J}_{h}^{n+1},\psi_{h}) =0,
\end{equation*}
and $(\text{div}\boldsymbol{u}_{h}^{n+1},\text{div}\boldsymbol{J}_{h}^{n+1})\in Q_{h}\times\Psi_{h}$ (see (\ref{Div-u}) and (\ref{Div-J})). Let $(q_{h},\psi_{h})=(\text{div}\boldsymbol{u}_{h}^{n+1},\text{div}\boldsymbol{J}_{h}^{n+1})$, we can naturally obtain $\text{div}\boldsymbol{u}_{h}^{n+1}=0$ and $\text{div}\boldsymbol{J}_{h}^{n+1}=0$.
\end{proof}
\subsection{A fully discrete second-order SAV scheme \label{section:second-order}}
In this section, we now construct a fully discrete second-order SAV virtual element approximation of equations (\ref{variational-eq}). The initial data $(\boldsymbol{u}_{h}^{0},\boldsymbol{J}_{h}^{0},s^{0})\in\boldsymbol{U}_{h}\times\boldsymbol{E}_{h}\times\mathbb{R}$ is consistent with the initial data given in subsection \ref{section:first-order} and $(\boldsymbol{u}_{h}^{1},\boldsymbol{J}_{h}^{1},p,\phi^{1},s^{1})\in\boldsymbol{U}_{h}\times\boldsymbol{E}_{h}\times Q_{h}\times\Psi_{h}\times\mathbb{R}$ is computed by the equations (\ref{eq:full-discrete-first-order}).

Assuming we have computed $(\boldsymbol{u}_{h}^{n},p_{h}^{n},\boldsymbol{J}_{h}^{n},\phi_{h}^{n},s^{n})\in \boldsymbol{U}_{h}\times
Q_{h}\times\boldsymbol{E}_{h}\times\Psi_{h}\times\mathbb{R}$ for $n>0$, and then we apply the backward Euler discretization in time and the virtual element method in spaces to deduce the numerical scheme reads as: find $\big(\boldsymbol{u}_{h}^{n+1},p_{h}^{n+1},\boldsymbol{J}_{h}^{n+1},\phi_{h}^{n+1},s^{n+1}\big)\in \boldsymbol{U}_{h}\times
Q_{h}\times\boldsymbol{E}_{h}\times\Psi_{h}\times\mathbb{R}$ such that
\begin{subequations}
	\label{eq:full-discrete-second-order}
	\begin{align}
\label{eq:full-discrete-second-order-u}
A_{1,h}(\delta_{t}^{2}\boldsymbol{u}_{h}^{n+1},\boldsymbol{v}_{h})+B_{h}(\boldsymbol{u}_{h}^{n+1},\boldsymbol{v}_{h})+C_{1,h}(\boldsymbol{v}_{h},p_{h}^{n+1})
+\frac{s^{n+1}}{R^{n+1}}\big(E_{s,h}(\boldsymbol{\tilde{u}}_{h}^{n+1};\boldsymbol{\tilde{u}}_{h}^{n+1},\boldsymbol{v}_{h})-D_{h}(\boldsymbol{\tilde{J}}_{h}^{n+1},\boldsymbol{v}_{h})\big)&=0,\\\label{eq:full-discrete-second-order-p}
C_{1,h}(\boldsymbol{u}_{h}^{n+1},q_{h}) &=0,\\\label{eq:full-discrete-second-order-J}
A_{2,h}(\boldsymbol{J}_{h}^{n+1},\boldsymbol{K}_{h})+C_{2,h}(\boldsymbol{K}_{h},\phi_{h}^{n+1})
+\frac{s^{n+1}}{R^{n+1}}D_{h}(\boldsymbol{K}_{h},\boldsymbol{\tilde{u}}_{h}^{n+1})&=0,\\\label{eq:full-discrete-second-order-phi}
C_{2,h}(\boldsymbol{J}_{h}^{n+1},\psi_{h}) &=0,\\ \label{eq:full-discrete-second-order-s} \delta_{t}^{2}s^{n+1}+s^{n+1}Q-\frac{1}{R^{n+1}}E_{s,h}(\boldsymbol{\tilde{u}}_{h}^{n+1};\boldsymbol{\tilde{u}}_{h}^{n+1},\boldsymbol{u}_{h}^{n+1})+\frac{1}{R^{n+1}}D_{h}(\boldsymbol{\tilde{J}}_{h}^{n+1},\boldsymbol{u}_{h}^{n+1})
-\frac{1}{R^{n+1}}D_{h}(\boldsymbol{J}_{h}^{n+1},\boldsymbol{\tilde{u}}_{h}^{n+1})& =0,
	\end{align}
\end{subequations}
for all $\big(\boldsymbol{v}_{h},q_{h},\boldsymbol{K}_{h},\psi_{h}\big)\in \boldsymbol{U}_{h}\times
Q_{h}\times\boldsymbol{E}_{h}\times
\Psi_{h}$  there hold. Subsequently, we establish the unconditionally energy stability of the formulation (\ref{eq:full-discrete-second-order}) as follows.
\begin{thm}
The formulation (\ref{eq:full-discrete-second-order}) is unconditionally energy stable in the sense that the following energy estimate
\begin{equation}
\label{Energy-second-order}
\delta_{t}^{n+1}\mathrm{E}^{n+1}\leq-\|\boldsymbol{u}_{h}^{n+1}\|_{B}^{2}-\|\boldsymbol{J}_{h}^{n+1}\|_{A_{2}}^{2}-Q|s^{n+1}|^2
\end{equation}
with
\begin{equation}
\mathrm{E}^{n+1} = \frac{1}{4}\|\boldsymbol{u}_{h}^{n+1}\|_{A_{1}}^{2} +\frac{1}{4}\|2\boldsymbol{u}_{h}^{n+1}-\boldsymbol{u}_{h}^{n}\|_{A_{1}}^{2}+ \frac{1}{4}|s^{n+1}|^2 + \frac{1}{4}|2s^{n+1}-s^{n}|^2.
\end{equation}
\end{thm}

\begin{proof}
Taking $\big(\boldsymbol{v}_{h},q_{h},\boldsymbol{K}_{h},\psi_{h}\big)=\big(\boldsymbol{u}_{h}^{n+1},p_{h}^{n+1},\boldsymbol{J}_{h}^{n+1},\phi_{h}^{n+1}\big)$ in equations (\ref{eq:full-discrete-second-order-u})-(\ref{eq:full-discrete-second-order-phi}) and applying the identity $2a(3a-4b+c)=a^2-b^2+(2a-b)^2-(2b-c)^2+(a-2b+c)^2$ and the definitions of norm in (\ref{define-norm}) lead to
\begin{equation}
\label{Energy-second-order-first}
	\begin{aligned}
\frac{1}{4\Delta t}\big(\|\boldsymbol{u}_{h}^{n+1}\|_{A_{1}}^{2}-\|\boldsymbol{u}_{h}^{n}\|_{A_{1}}^{2}
+\|2\boldsymbol{u}_{h}^{n+1}-\boldsymbol{u}_{h}^{n}\|_{A_{1}}^{2}
-\|2\boldsymbol{u}_{h}^{n}-\boldsymbol{u}_{h}^{n-1}\|_{A_{1}}^{2}
+\|\boldsymbol{u}_{h}^{n+1}-2\boldsymbol{u}_{h}^{n}+\boldsymbol{u}_{h}^{n-1}\|_{A_{1}}^{2}\big)
\\+\|\boldsymbol{u}_{h}^{n+1}\|_{B}^{2}+\|\boldsymbol{J}_{h}^{n+1}\|_{A_{2}}^{2}
+\frac{s^{n+1}}{R^{n+1}}E_{s,h}(\boldsymbol{\tilde{u}}_{h}^{n};\boldsymbol{\tilde{u}}_{h}^{n},\boldsymbol{u}_{h}^{n+1})
-\frac{s^{n+1}}{R^{n+1}}D_{h}(\boldsymbol{\tilde{J}}_{h}^{n},\boldsymbol{u}_{h}^{n+1})
+\frac{s^{n+1}}{R^{n+1}}D_{h}(\boldsymbol{J}_{h}^{n+1},\boldsymbol{\tilde{u}}_{h}^{n})&=0.\\
	\end{aligned}
\end{equation}
Multiplying (\ref{eq:full-discrete-second-order-s}) by $s^{n+1}$ leads to
\begin{equation}
\label{Energy-second-order-second}
	\begin{aligned}
\frac{1}{4\Delta t}\big(|s^{n+1}|^2-|s^{n}|^2+|2s^{n+1}-s^{n}|^2-|2s^{n}-s^{n-1}|^{2}+|s^{n+1}-2s^{n}+s^{n-1}|^{2}\big)+Q|s^{n+1}|^2&\\
-\frac{s^{n+1}}{R^{n+1}}E_{s,h}(\boldsymbol{\tilde{u}}_{h}^{n};\boldsymbol{\tilde{u}}_{h}^{n},\boldsymbol{u}_{h}^{n+1})
+\frac{s^{n+1}}{R^{n+1}}D_{h}(\boldsymbol{\tilde{J}}_{h}^{n},\boldsymbol{u}_{h}^{n+1})
-\frac{s^{n+1}}{R^{n+1}}D_{h}(\boldsymbol{J}_{h}^{n+1},\boldsymbol{\tilde{u}}_{h}^{n})&=0.
	\end{aligned}
\end{equation}
Combining (\ref{Energy-second-order-first}) with (\ref{Energy-second-order-second})  results in obtaining
\begin{equation*}
	\begin{aligned}
\frac{1}{4\Delta t}\big(\|\boldsymbol{u}_{h}^{n+1}\|_{A_{1}}^{2}-\|\boldsymbol{u}_{h}^{n}\|_{A_{1}}^{2}
+\|2\boldsymbol{u}_{h}^{n+1}-\boldsymbol{u}_{h}^{n}\|_{A_{1}}^{2}
-\|2\boldsymbol{u}_{h}^{n}-\boldsymbol{u}_{h}^{n-1}\|_{A_{1}}^{2}
+\|\boldsymbol{u}_{h}^{n+1}-2\boldsymbol{u}_{h}^{n}+\boldsymbol{u}_{h}^{n-1}\|_{A_{1}}^{2}\big)&\\
+\frac{1}{4\Delta t}\big(|s^{n+1}|^2-|s^{n}|^2+|2s^{n+1}-s^{n}|^2-|2s^{n}-s^{n-1}|^{2}+|s^{n+1}-2s^{n}+s^{n-1}|^{2}\big)&\\
+\|\boldsymbol{u}_{h}^{n+1}\|_{B}^{2}+\|\boldsymbol{J}_{h}^{n+1}\|_{A_{2}}^{2}+Q|s^{n+1}|^2&=0.
	\end{aligned}
\end{equation*}
This yields the conclusion (\ref{Energy-second-order}).
\end{proof}

\begin{prop}
Let $\big(\boldsymbol{u}_{h}^{n+1},p_{h}^{n+1},\boldsymbol{J}_{h}^{n+1},\phi_{h}^{n+1},s^{n+1}\big)\in \boldsymbol{U}_{h}\times
Q_{h}\times\boldsymbol{E}_{h}\times\Psi_{h}\times\mathbb{R}$ be the solution to the equations (\ref{eq:full-discrete-second-order}), the scheme fully satisfies the properties of mass conservation and charge conservation. Namely, $\text{div}\boldsymbol{u}_{h}^{n+1}=0$ and $\text{div}\boldsymbol{J}_{h}^{n+1}=0$.
\end{prop}
\begin{proof}
The proof procedure is consistent with Proposition \ref{Prop:first-divergence free}.
\end{proof}
\section{Numerical experiments}
In this section, we consider a example to numerically verify the theoretical trend of all the error for a three dimensional inductionless MHD problem.  The numerical experiment is implemented on the software MATLAB.

Now, we present the practical implementation of the first-order fully discrete SAV virtual element formulation (\ref{eq:full-discrete-first-order}).
For $n=0$, we have discussed the implementation in the subsection \ref{section:first-order}. Thus, given the previous solution $(\boldsymbol{u}_{h}^{n},p_{h}^{n},\boldsymbol{J}_{h}^{n},\phi_{h}^{n},s^{n})\in \boldsymbol{U}_{h}\times
Q_{h}\times\boldsymbol{E}_{h}\times
\Psi_{h}\times\mathbb{R}$, the solution $(\boldsymbol{u}_{h}^{n+1},p_{h}^{n+1},\boldsymbol{J}_{h}^{n+1},\phi_{h}^{n+1},s^{n+1})\in \boldsymbol{U}_{h}\times
Q_{h}\times\boldsymbol{E}_{h}\times
\Psi_{h}\times\mathbb{R}$ with $n\geq0$ is calculated as follows:
\begin{itemize}
\item[$\bullet$] \texttt{Stokes-type equations}
\begin{equation}
\label{first-order-implementation-step11}
\left\{
\begin{aligned}
&\text{find}~(\boldsymbol{u}_{1,h}^{n+1},p_{1,h}^{n+1})\in \boldsymbol{U}_{h}\times
Q_{h},~\text{such that}\\
&\frac{A_{1,h}(\boldsymbol{u}_{1,h}^{n+1},\boldsymbol{v}_{h})
}{\Delta t}
+B_{h}(\boldsymbol{u}_{1,h}^{n+1},\boldsymbol{v}_{h})+C_{1,h}(\boldsymbol{v}_{h},p_{1,h}^{n+1}) =\frac{A_{1,h}(\boldsymbol{u}_{h}^{n},\boldsymbol{v}_{h})}{\Delta t} \quad\forall\boldsymbol{v}_{h}\in\boldsymbol{U}_{h},\\
&C_{1,h}(\boldsymbol{u}_{1,h}^{n+1},q_{h}) =0 \quad\forall q_{h}\in Q_{h},
\end{aligned}
\right.
\end{equation}
and
\begin{equation}
\label{first-order-implementation-step12}
\left\{
\begin{aligned}
&\text{find}~(\boldsymbol{u}_{2,h}^{n+1},p_{2,h}^{n+1})\in \boldsymbol{U}_{h}\times
Q_{h},~\text{such that}\\
&\frac{A_{1,h}(\boldsymbol{u}_{2,h}^{n+1},\boldsymbol{v}_{h})}{\Delta t}+ B_{h}(\boldsymbol{u}_{2,h}^{n+1},\boldsymbol{v}_{h})+C_{1,h}(\boldsymbol{v}_{h},p_{2,h}^{n+1}) =D_{h}(\boldsymbol{J}_{h}^{n},\boldsymbol{v}_{h})-E_{s,h}(\boldsymbol{u}_{h}^{n};\boldsymbol{u}_{h}^{n},\boldsymbol{v}_{h})\quad\forall\boldsymbol{v}_{h}\in\boldsymbol{U}_{h},\\
&C_{1,h}(\boldsymbol{u}_{2,h}^{n+1},q_{h}) =0\quad\forall q_{h}\in Q_{h},
\end{aligned}
\right.
\end{equation}
\item[$\bullet$] \texttt{Mixed-Poisson equations}
\begin{equation}
\label{first-order-implementation-step21}
\left\{
\begin{aligned}
&\text{find}~(\boldsymbol{J}_{1,h}^{n+1},\phi_{1,h}^{n+1})\in \boldsymbol{E}_{h}\times
\Psi_{h},~\text{such that}\\
&\frac{A_{2,h}(\boldsymbol{J}_{1,h}^{n+1},\boldsymbol{K}_{h})
}{\Delta t}
+C_{2,h}(\boldsymbol{K}_{h},\phi_{1,h}^{n+1}) =0 \quad\forall\boldsymbol{K}_{h}\in\boldsymbol{E}_{h},\\
&C_{2,h}(\boldsymbol{J}_{1,h}^{n+1},\psi_{h}) =0 \quad\forall \psi_{h}\in \Psi_{h},
\end{aligned}
\right.
\end{equation}
and
\begin{equation}
\label{first-order-implementation-step22}
\left\{
\begin{aligned}
&\text{find}~(\boldsymbol{J}_{2,h}^{n+1},\phi_{2,h}^{n+1})\in \boldsymbol{E}_{h}\times
\Psi_{h},~\text{such that}\\
&\frac{A_{2,h}(\boldsymbol{J}_{2,h}^{n+1},\boldsymbol{K}_{h})}{\Delta t}+C_{2,h}(\boldsymbol{K}_{h},\phi_{2,h}^{n+1}) =-D_{h}(\boldsymbol{K}_{h},\boldsymbol{u}_{h}^{n})\quad\forall\boldsymbol{K}_{h}\in\boldsymbol{E}_{h},\\
&C_{2,h}(\boldsymbol{J}_{2,h}^{n+1},\psi_{h}) =0\quad\forall \psi_{h}\in \Psi_{h},
\end{aligned}
\right.
\end{equation}
\item[$\bullet$]\texttt{linear algebra equation}
\begin{equation}
\label{first-order-implementation-step2}
\left\{
\begin{aligned}
&\text{find}~S^{n+1}\in \mathbb{R},~\text{such that}\\
&\big(\frac{R^{n+1}}{\Delta t}+R^{n+1}Q-\frac{\mathcal{A}_{2}}{R^{n+1}}\big)S^{n+1}=\frac{\mathcal{A}_{1}}{R^{n+1}}+\frac{1}{\Delta t}s^{n},
\end{aligned}
\right.
\end{equation}
where $\mathcal{A}_{i}~(i=1,2)$ are denoted as
\begin{equation*}
\mathcal{A}_{i}= E_{s,h}(\boldsymbol{u}_{h}^{n};\boldsymbol{u}_{h}^{n},\boldsymbol{u}_{i,h}^{n+1})+D_{h}(\boldsymbol{J}_{i,h}^{n+1},\boldsymbol{u}_{h}^{n})-D_{h}(\boldsymbol{J}_{h}^{n},\boldsymbol{u}_{i,h}^{n+1})
\end{equation*}
\item[$\bullet$]\texttt{the solution}
\begin{equation}
\label{first-order-implementation-step3}
\left\{
\begin{aligned}
&(\boldsymbol{u}_{h}^{n+1},p^{n+1},\boldsymbol{J}_{h}^{n+1},\phi^{n+1})=(\boldsymbol{u}_{1,h}^{n+1},p_{1,h}^{n+1},\boldsymbol{J}_{1,h}^{n+1},\phi_{1,h}^{n+1})+S^{n+1}(\boldsymbol{u}_{2,h}^{n+1},p_{2,h}^{n+1},\boldsymbol{J}_{2,h}^{n+1},\phi_{2,h}^{n+1}),\\ &s^{n+1}=S^{n+1}R^{n+1}.
\end{aligned}
\right.
\end{equation}
\end{itemize}

Next, we present the practical implementation of the second-order fully discrete SAV virtual element formulation (\ref{eq:full-discrete-second-order}).
For $n=0,1$, we have discussed the implementation in the subsection \ref{section:second-order}. Thus, given the previous solution $(\boldsymbol{u}_{h}^{n},p_{h}^{n},\boldsymbol{J}_{h}^{n},\phi_{h}^{n},s^{n})\in \boldsymbol{U}_{h}\times
Q_{h}\times\boldsymbol{E}_{h}\times
\Psi_{h}\times\mathbb{R}$, the solution $(\boldsymbol{u}_{h}^{n+1},p_{h}^{n+1},\boldsymbol{J}_{h}^{n+1},\phi_{h}^{n},s^{n+1})\in \boldsymbol{U}_{h}\times
Q_{h}\times\boldsymbol{E}_{h}\times
\Psi_{h}\times\mathbb{R}$ with $n\geq1$ is calculated as follows:
\begin{itemize}
\item[$\bullet$] \texttt{Stokes-type equations}
\begin{equation}
\label{second-order-implementation-step11}
\left\{
\begin{aligned}
&\text{find}~(\boldsymbol{u}_{1,h}^{n+1},p_{1,h}^{n+1})\in \boldsymbol{U}_{h}\times
Q_{h},~\text{such that}\\
&\frac{3A_{1,h}(\boldsymbol{u}_{1,h}^{n+1},\boldsymbol{v}_{h})
}{2\Delta t}
+B_{h}(\boldsymbol{u}_{1,h}^{n+1},\boldsymbol{v}_{h})+C_{1,h}(\boldsymbol{v}_{h},p_{1,h}^{n+1}) =\frac{4A_{1,h}(\boldsymbol{u}_{h}^{n},\boldsymbol{v}_{h})}{2\Delta t}
-\frac{A_{1,h}(\boldsymbol{u}_{h}^{n-1},\boldsymbol{v}_{h})}{2\Delta t} \quad\forall\boldsymbol{v}_{h}\in\boldsymbol{U}_{h},\\
&C_{1,h}(\boldsymbol{u}_{1,h}^{n+1},q_{h}) =0 \quad\forall q_{h}\in Q_{h},
\end{aligned}
\right.
\end{equation}
and
\begin{equation}
\label{second-order-implementation-step12}
\left\{
\begin{aligned}
&\text{find}~(\boldsymbol{u}_{2,h}^{n+1},p_{2,h}^{n+1})\in \boldsymbol{U}_{h}\times
Q_{h},~\text{such that}\\
&\frac{3A_{1,h}(\boldsymbol{u}_{2,h}^{n+1},\boldsymbol{v}_{h})}{2\Delta t}+ B_{h}(\boldsymbol{u}_{2,h}^{n+1},\boldsymbol{v}_{h})+C_{1,h}(\boldsymbol{v}_{h},p_{2,h}^{n+1}) =D_{h}(\boldsymbol{\tilde{J}}_{h}^{n+1},\boldsymbol{v}_{h})-E_{s,h}(\boldsymbol{\tilde{u}}_{h}^{n+1};\boldsymbol{\tilde{u}}_{h}^{n+1},\boldsymbol{v}_{h})\quad\forall\boldsymbol{v}_{h}\in\boldsymbol{U}_{h},\\
&C_{1,h}(\boldsymbol{u}_{2,h}^{n+1},q_{h}) =0\quad\forall q_{h}\in Q_{h},
\end{aligned}
\right.
\end{equation}
\item[$\bullet$] \texttt{Mixed-Poisson equations}
\begin{equation}
\label{second-order-implementation-step21}
\left\{
\begin{aligned}
&\text{find}~(\boldsymbol{J}_{1,h}^{n+1},\phi_{1,h}^{n+1})\in \boldsymbol{E}_{h}\times
\Psi_{h},~\text{such that}\\
&\frac{A_{2,h}(\boldsymbol{J}_{1,h}^{n+1},\boldsymbol{K}_{h})
}{\Delta t}
+C_{2,h}(\boldsymbol{K}_{h},\phi_{1,h}^{n+1}) =0 \quad\forall\boldsymbol{K}_{h}\in\boldsymbol{E}_{h},\\
&C_{2,h}(\boldsymbol{J}_{1,h}^{n+1},\psi_{h}) =0 \quad\forall \psi_{h}\in \Psi_{h},
\end{aligned}
\right.
\end{equation}
and
\begin{equation}
\label{second-order-implementation-step22}
\left\{
\begin{aligned}
&\text{find}~(\boldsymbol{J}_{2,h}^{n+1},\phi_{2,h}^{n+1})\in \boldsymbol{E}_{h}\times
\Psi_{h},~\text{such that}\\
&\frac{A_{2,h}(\boldsymbol{J}_{2,h}^{n+1},\boldsymbol{K}_{h})}{\Delta t}+C_{2,h}(\boldsymbol{K}_{h},\phi_{2,h}^{n+1}) =-D_{h}(\boldsymbol{K}_{h},\boldsymbol{\tilde{u}}_{h}^{n+1})\quad\forall\boldsymbol{K}_{h}\in\boldsymbol{E}_{h},\\
&C_{2,h}(\boldsymbol{J}_{2,h}^{n+1},\psi_{h}) =0\quad\forall \psi_{h}\in \Psi_{h},
\end{aligned}
\right.
\end{equation}
\item[$\bullet$]\texttt{linear algebra equation}
\begin{equation}
\label{second-order-implementation-step2}
\left\{
\begin{aligned}
&\text{find}~S^{n+1}\in \mathbb{R},~\text{such that}\\
&\big(\frac{3R^{n+1}}{2\Delta t}+R^{n+1}Q-\frac{\mathcal{A}_{2}}{R^{n+1}}\big)S^{n+1}=\frac{\mathcal{A}_{1}}{R^{n+1}}+\frac{4}{2\Delta t}s^{n}-\frac{1}{2\Delta t}s^{n-1},
\end{aligned}
\right.
\end{equation}
where $\mathcal{A}_{i}~(i=1,2)$ are denoted as
\begin{equation*}
\mathcal{A}_{i}= E_{s,h}(\boldsymbol{\tilde{u}}_{h}^{n+1};\boldsymbol{\tilde{u}}_{h}^{n+1},\boldsymbol{u}_{i,h}^{n+1})+D_{h}(\boldsymbol{J}_{i,h}^{n+1},\boldsymbol{\tilde{u}}_{h}^{n+1})-D_{h}(\boldsymbol{\tilde{J}}_{h}^{n+1},\boldsymbol{u}_{i,h}^{n+1})
\end{equation*}
\item[$\bullet$]\texttt{the solution}
\begin{equation}
\label{second-order-implementation-step3}
\left\{
\begin{aligned}
&(\boldsymbol{u}_{h}^{n+1},p^{n+1},\boldsymbol{J}_{h}^{n+1},\phi^{n+1})=(\boldsymbol{u}_{1,h}^{n+1},p_{1,h}^{n+1},\boldsymbol{J}_{1,h}^{n+1},\phi_{1,h}^{n+1})+S^{n+1}(\boldsymbol{u}_{2,h}^{n+1},p_{2,h}^{n+1},\boldsymbol{J}_{2,h}^{n+1},\phi_{2,h}^{n+1}),\\ &s^{n+1}=S^{n+1}R^{n+1}.
\end{aligned}
\right.
\end{equation}
\end{itemize}
In sections \ref{test1}, we use three types of decompositions of the standard domain $\Omega=[0,1]^{3}$ (see FIGURE.\ref{polyhedral meshes}):
\begin{itemize}
\item[$\bullet$]\texttt{Dtp}: Distorted triangular prism meshes; see FIGURE.\ref{polyhedral meshes}(a),
\item[$\bullet$]\texttt{Cube}: Structured cubes meshes; see FIGURE.\ref{polyhedral meshes}(b),
\item[$\bullet$]\texttt{Simple-Voronoi}: Simple-Voronoi polyhedral meshes; see FIGURE.\ref{polyhedral meshes}(c).
\end{itemize}
\begin{figure}
\centering
\subfloat[\texttt{Dtp}]{
\label{fig.mesh3da}
\includegraphics[width=5cm]{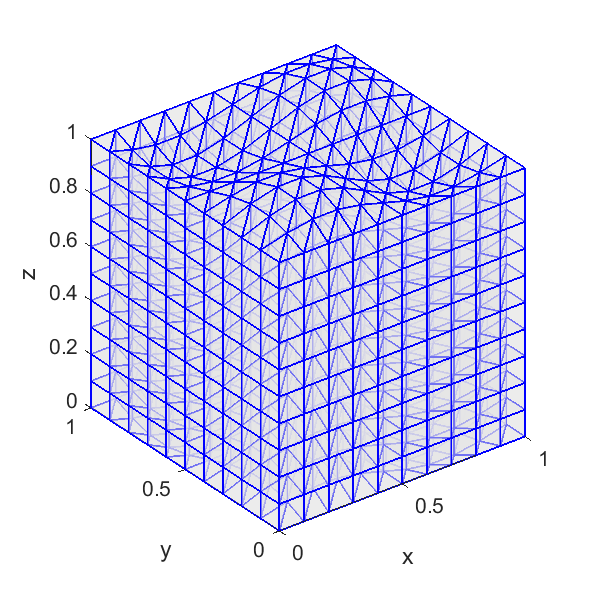}}
\hspace{-0.50cm}
\subfloat[\texttt{Cube}]{
\label{fig.mesh3db}
\includegraphics[width=5cm]{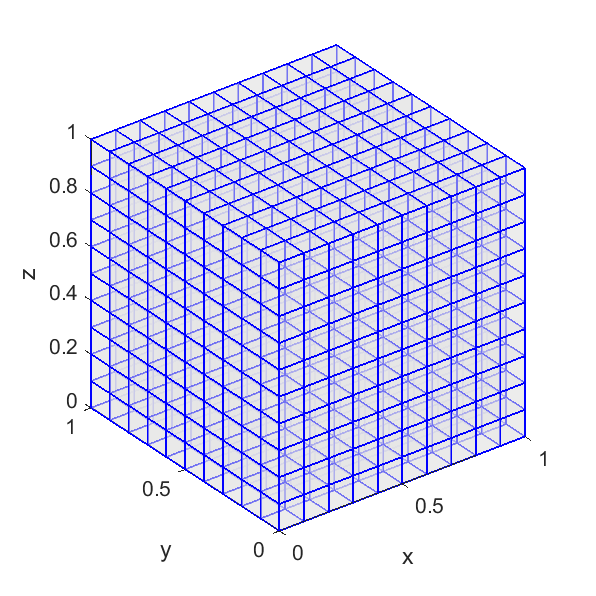}}
\hspace{-0.50cm}
\subfloat[\texttt{Simple-Voronoi}]{
\label{fig.mesh3dc}
\includegraphics[width=5cm]{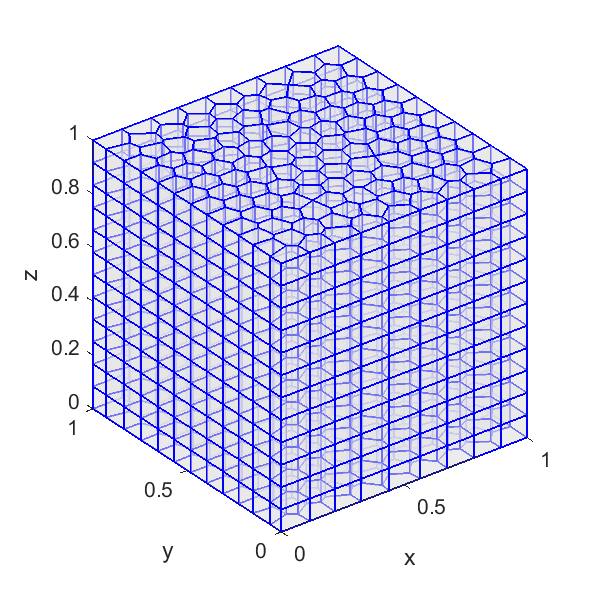}}
\caption{A sample of the used meshes for the examples \ref{test1}.}
\label{polyhedral meshes}
\end{figure}
Then, we associate with each mesh a mesh-size
\begin{equation*}
h:=(\frac{|\mathrm{\Omega}|}{N_{K}})^{\frac{1}{3}},
\end{equation*}
where $N_{K}$ is the number of polyhedrons $K$ in the mesh.
Under the framework of VEM, since the discrete velocity $\boldsymbol{u}_{h}^{N}$ and current density $\boldsymbol{J}_{h}^{N}$ are virtual, we measure approximate absolute errors involving suitable polynomial projections. On the other hand, the discrete pressure $p_{h}^{N}$ and electric potential $\phi_{h}^{N}$ are piecewise linear polynomials over $\mathcal{T}_{h}$. Thence, we compute the following error quantities:
\begin{align*}
e_{\boldsymbol{u}}^{N}&=\sqrt{\sum_{K\in\mathcal{T}_{h}}\|\nabla\boldsymbol{u}(t^{N})-\nabla\mathrm{\Pi}_{k_{u}}^{\nabla,K}\boldsymbol{u}_{h}^{N}\|_{0,K}},\quad
e_{\boldsymbol{J}}^{N}=\sqrt{\sum_{K\in\mathcal{T}_{h}}\|\boldsymbol{J}(t^{N})-\mathrm{\Pi}_{k_{J}}^{0,K}\boldsymbol{J}_{h}^{N}\|_{0,K}},\\
e_{p}^{N} &= \|p(t^{N})-p_{h}^{N}\|_{0,\mathrm{\Omega}},\quad
e_{\phi}^{N} = \|\phi(t^{N})-\phi_{h}^{N}\|_{0,\mathrm{\Omega}},\quad e_{s}^{N} = |s(t^{N})-s_{h}^{N}|.\\
\end{align*}
\subsection{Example 1\label{test1}}
Let the computational domain $\mathrm{\Omega}=[0,1]^{3}$, $T=1$, the parameters $Re = 1$, $\kappa=1$. We test the convergence orders of the time discretization and the fully discrezation. On the one hand, to test the convergence orders of the time discretization, we choose the analytical solution to the three dimensional inductionless MHD equations with load terms as follows:
\begin{align*}
\boldsymbol{u}
=\left[\begin{aligned}
x_{3}\sin(t)\\
x_{3}\\
0
\end{aligned}
\right],\quad
\boldsymbol{J}
=\left[\begin{aligned}
\cos(t)\\
t^{2}\\
0
\end{aligned}
\right],\quad
p=\phi=0,
\end{align*}
The mesh sizes of the three types of meshes are given as $h_{D} = 1.9843e-01$, $h_{C} = 2.5000e-01$, and $h_{S}=2.5000e-01$, respectively. The results presented in FIGURE.\ref{the first-order time discretization} show that the convergence rates of all variables achieve first order accuracy, which is consistent with the expected convergence rates $O(\Delta t)$. The results presented in FIGURE.\ref{the second-order time discretization} show that the convergence rates of all variables achieve second order accuracy, which is consistent with the expected convergence rates $O((\Delta t)^{2})$.
\begin{figure}
\centering
\subfloat[\texttt{Dtp}$-h_{T}$]{
\label{fig.mesh3da_first_time}
\includegraphics[width=5.5cm]{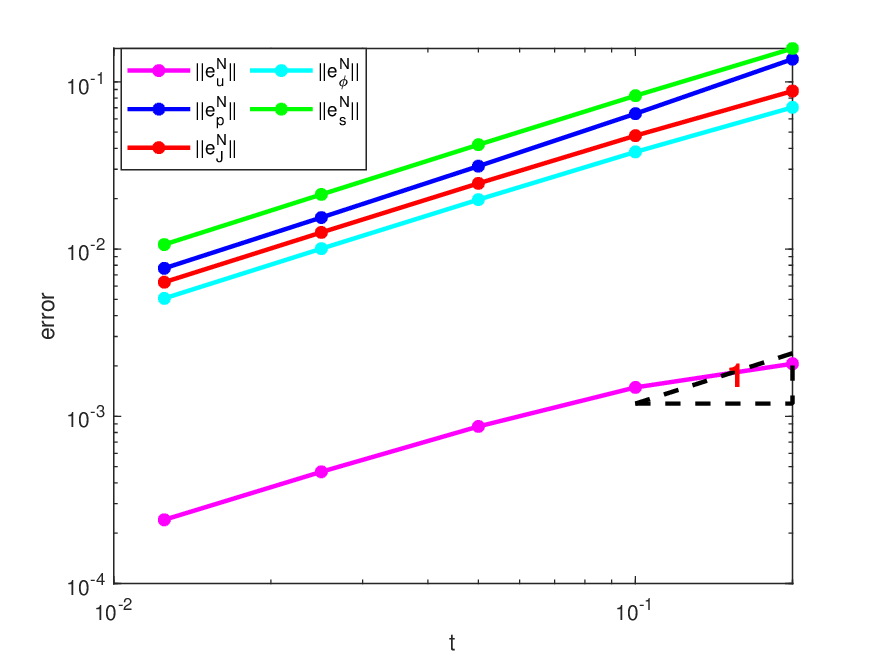}}
\hspace{-0.50cm}
\subfloat[\texttt{Cube}$-h_{C}$]{
\label{fig.mesh3db_first_time}
\includegraphics[width=5.5cm]{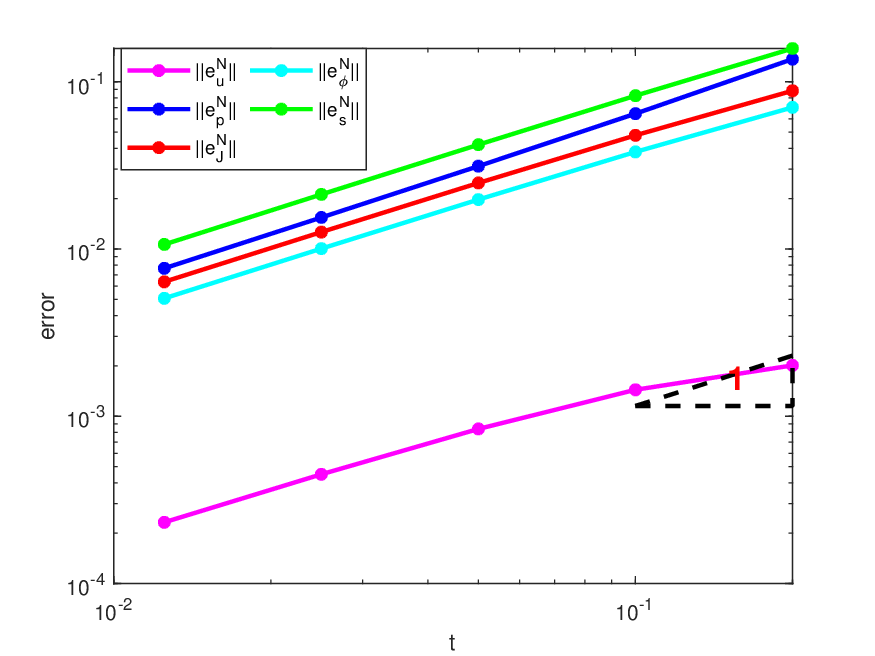}}
\hspace{-0.50cm}
\subfloat[\texttt{Simple-Voronoi}$-h_{V}$]{
\label{fig.mesh3dc_first_time}
\includegraphics[width=5.5cm]{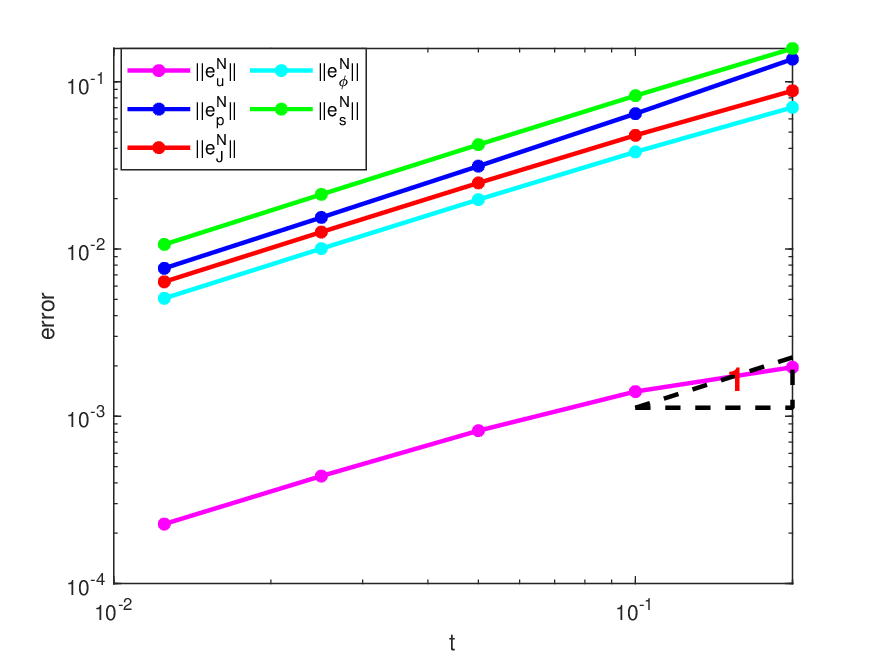}} 
\caption{Convergence rates of the first-order time discretization at $t=1$.}
\label{the first-order time discretization}
\end{figure}
\begin{table}[htp]\footnotesize
\centering
\setlength{\tabcolsep}{1mm}{
\caption{Divergence-free results for the first-order time discretization at $t=1$.\label{table:Divergence-free-time-discrete-first}}
\begin{tabular*}{14cm}{@{\extracolsep{\fill}} c c c | c c c| c c c}
\hline
&\multicolumn{2}{c|}{\texttt{Dtp}}& &
\multicolumn{2}{c|}{\texttt{Cube}}& &
\multicolumn{2}{c}{\texttt{Simple-Voronoi}} \\
\hline
 $\Delta t$ & $\|\text{div}\boldsymbol{u}_{h}^{N}\|_{0}$  & $\|\text{div}\boldsymbol{J}_{h}^{N}\|_{0}$ &$\Delta t$ &$\|\text{div}\boldsymbol{u}_{h}^{N}\|_{0}$& $\|\text{div}\boldsymbol{J}_{h}^{N}\|_{0}$&$\Delta t$ &$\|\text{div}\boldsymbol{u}_{h}^{N}\|_{0}$ &$\|\text{div}\boldsymbol{J}_{h}^{N}\|_{0}$\\
 \hline
2.0000e-01 & 5.3166e-16 & 5.7501e-15 &2.0000e-01 & 3.2307e-16 & 4.3805e-15 & 2.0000e-01 & 3.5919e-16 & 4.4610e-15\\
1.0000e-01 & 4.5985e-16 & 5.7240e-15 &1.0000e-01 & 3.9852e-16 & 4.4316e-15 & 1.0000e-01 & 3.4773e-16 & 4.5498e-15\\
5.0000e-02 & 4.8857e-16 & 5.8068e-15 &5.0000e-02 & 2.9806e-16 & 4.2900e-15 & 5.0000e-02 & 3.5333e-16 & 4.7274e-15\\
2.5000e-02 & 4.9969e-16 & 5.7903e-15 &2.5000e-02 & 3.4155e-16 & 4.2666e-15 & 2.5000e-02 & 3.4442e-16 & 4.3970e-15\\
1.2500e-02 & 5.5949e-16 & 5.5778e-15 &1.2500e-02 & 3.4429e-16 & 4.2531e-15 & 1.2500e-02 & 3.5727e-16 & 4.5223e-15\\
\hline
\end{tabular*}}
\end{table}
\begin{figure}
\centering
\subfloat[\texttt{Dtp}$-h_{T}$]{
\label{fig.mesh3da_second_time}
\includegraphics[width=5.5cm]{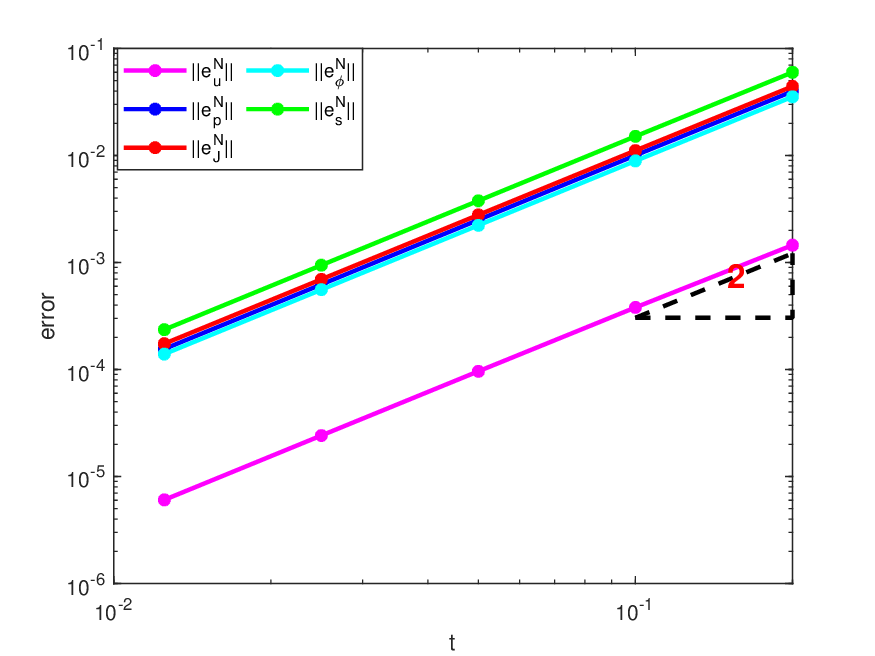}}
\hspace{-0.50cm}
\subfloat[\texttt{Cube}$-h_{C}$]{
\label{fig.mesh3db_second_time}
\includegraphics[width=5.5cm]{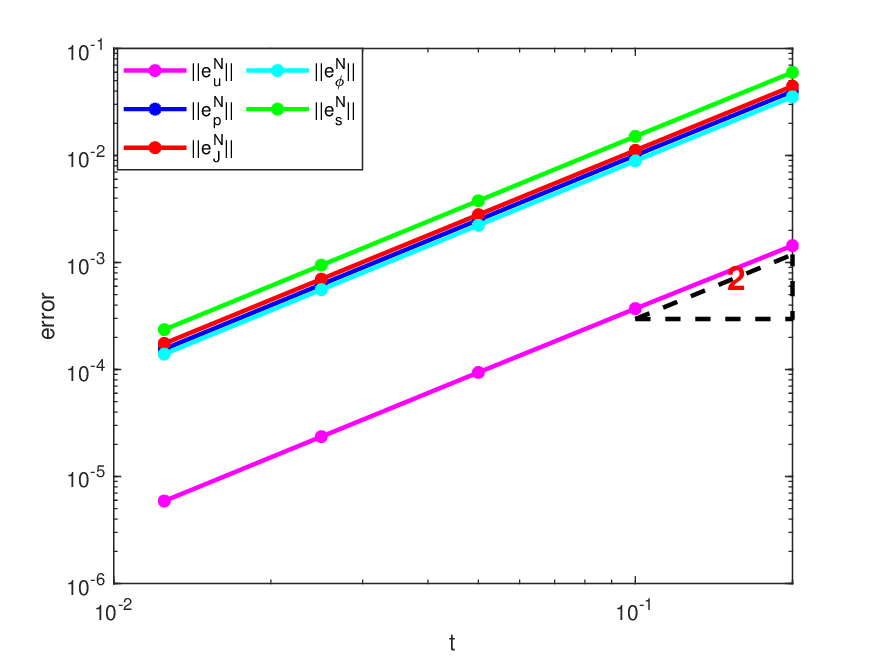}}
\hspace{-0.50cm}
\subfloat[\texttt{Simple-Voronoi}$-h_{V}$]{
\label{fig.mesh3dc_second_time}
\includegraphics[width=5.5cm]{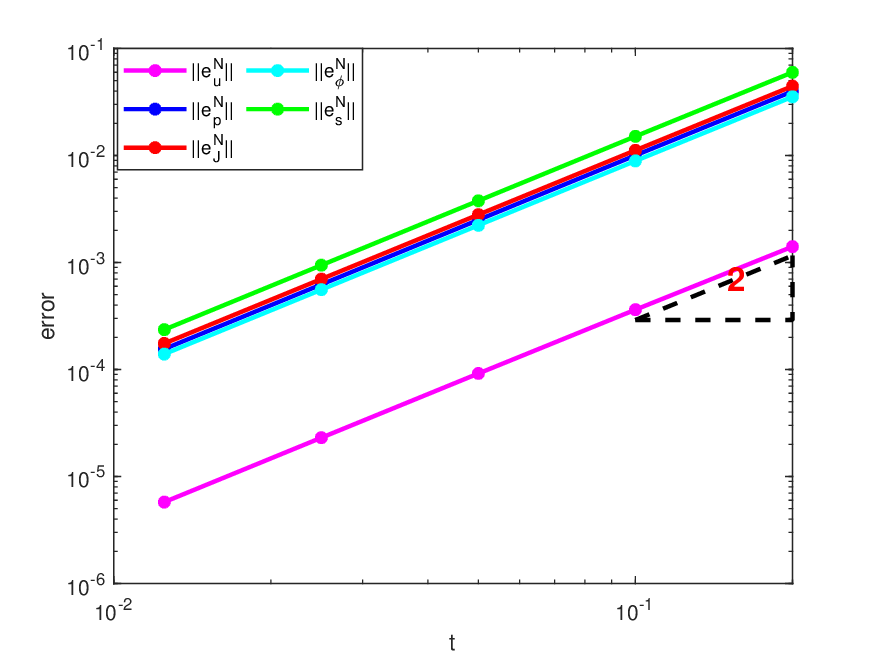}} 
\caption{Convergence rates of the second-order time discretization at $t=1$.}
\label{the second-order time discretization}
\end{figure}
\begin{table}[htp]\footnotesize
\centering
\setlength{\tabcolsep}{1mm}{
\caption{Divergence-free results for the second-order time discretization at $t=1$.\label{table:Divergence-free-time-discrete-second}}
\begin{tabular*}{14cm}{@{\extracolsep{\fill}} c c c | c c c| c c c}
\hline
&\multicolumn{2}{c|}{\texttt{Dtp}}& &
\multicolumn{2}{c|}{\texttt{Cube}}& &
\multicolumn{2}{c}{\texttt{Simple-Voronoi}} \\
\hline
 $\Delta t$ & $\|\text{div}\boldsymbol{u}_{h}^{N}\|_{0}$  & $\|\text{div}\boldsymbol{J}_{h}^{N}\|_{0}$ &$\Delta t$ &$\|\text{div}\boldsymbol{u}_{h}^{N}\|_{0}$& $\|\text{div}\boldsymbol{J}_{h}^{N}\|_{0}$&$\Delta t$ &$\|\text{div}\boldsymbol{u}_{h}^{N}\|_{0}$ &$\|\text{div}\boldsymbol{J}_{h}^{N}\|_{0}$\\
 \hline
2.0000e-01 & 4.7807e-16 & 5.7669e-15 &2.0000e-01 & 3.5411e-16 & 4.4411e-15 & 2.0000e-01 & 3.9349e-16 & 4.5812e-15\\
1.0000e-01 & 5.3165e-16 & 5.7562e-15 &1.0000e-01 & 3.6550e-16 & 4.2818e-15 & 1.0000e-01 & 3.6207e-16 & 4.6888e-15\\
5.0000e-02 & 4.8159e-16 & 5.6531e-15 &5.0000e-02 & 3.2546e-16 & 4.4019e-15 & 5.0000e-02 & 3.3791e-16 & 4.5448e-15\\
2.5000e-02 & 4.9765e-16 & 5.5286e-15 &2.5000e-02 & 3.5570e-16 & 4.3388e-15 & 2.5000e-02 & 3.1380e-16 & 4.5415e-15\\
1.2500e-02 & 4.7018e-16 & 5.6786e-15 &1.2500e-02 & 3.5816e-16 & 4.3915e-15 & 1.2500e-02 & 3.4920e-16 & 4.5984e-15\\
\hline
\end{tabular*}}
\end{table}

 On the other hand, to validate the convergence orders of the fully discretization, we choose the mesh sizes $h_{i}$ $(i=D,C,S)$ and the time step simultaneously satisfy $\Delta t = h_{i}^{2}$ and $\Delta t = h_{i}$ for the first-order scheme and the second-order scheme, respectively. The following analytical solution to the three dimensional inductionless MHD equations with load terms as follows:
 \begin{align*}
\boldsymbol{u}
=&\left[\begin{aligned}
\sin(\pi x_{1})\cos(\pi x_{2})\cos(\pi x_{3})\exp(-t)\\
\cos(\pi x_{1})\sin(\pi x_{2})\cos(\pi x_{3})\exp(-t)\\
-2\cos(\pi x_{1})\cos(\pi x_{2})\sin(\pi x_{3})\exp(-t)
\end{aligned}
\right],\quad
\boldsymbol{J}
=\left[\begin{aligned}
x_{2}(1-x_{2})x_{3}(1-x_{3})\exp(-t)\\
x_{1}(1-x_{1})x_{3}(1-x_{3})\exp(-t)\\
x_{1}(1-x_{1})x_{2}(1-x_{2})\exp(-t)
\end{aligned}
\right],\\
p=&-\pi\cos(\pi x_{1})\cos(\pi x_{2})\sin(\pi x_{3})\exp(-t),~~~~~~~\phi=(x^2+y^2+z^2-1)\exp(-t),
\end{align*}
The mesh sizes of the three types of meshes are given as $h_{D} = [3.9685e-01,1.9843e-01,1.3228e-01,9.9213e-02,7.9370e-02]$, $h_{C} = [ 5.0000e-01,2.5000e-01,1.6667e-01,1.2500e-01,1.0000e-01]$, and $h_{S}=[5.0000e-01,2.5000e-01,1.6466e-01,1.3069e-01, 8.6670e-02]$, respectively. The results presented in FIGURES.\ref{the first-order fully discretization}-\ref{the second-order fully discretization} show that the convergent orders of all variables achieve second-order accuracy, which are consistent with the expected convergence rates $O(h^{2})$. To sum up, the numerical results are consistent with the convergence order of results we expect for the proposed formulations (\ref{eq:full-discrete-first-order}) and (\ref{eq:full-discrete-second-order}). In particularly, we note that the final discrete velocity and the discrete current density are both $pointwise$ divergence-free in Tables \ref{table:Divergence-free-full-discrete-first}-\ref{table:Divergence-free-full-discrete-second}.
\begin{figure}
\centering
\subfloat[\texttt{Dtp}$-h_{T}$]{
\label{fig.mesh3da_first_mixed}
\includegraphics[width=5.5cm]{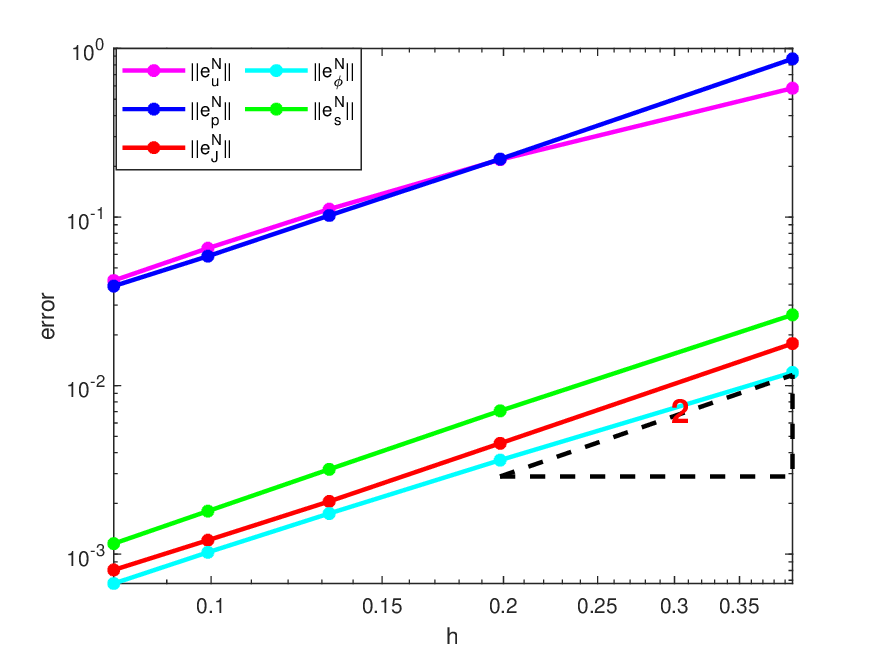}}
\hspace{-0.50cm}
\subfloat[\texttt{Cube}$-h_{C}$]{
\label{fig.mesh3db_first_mixed}
\includegraphics[width=5.5cm]{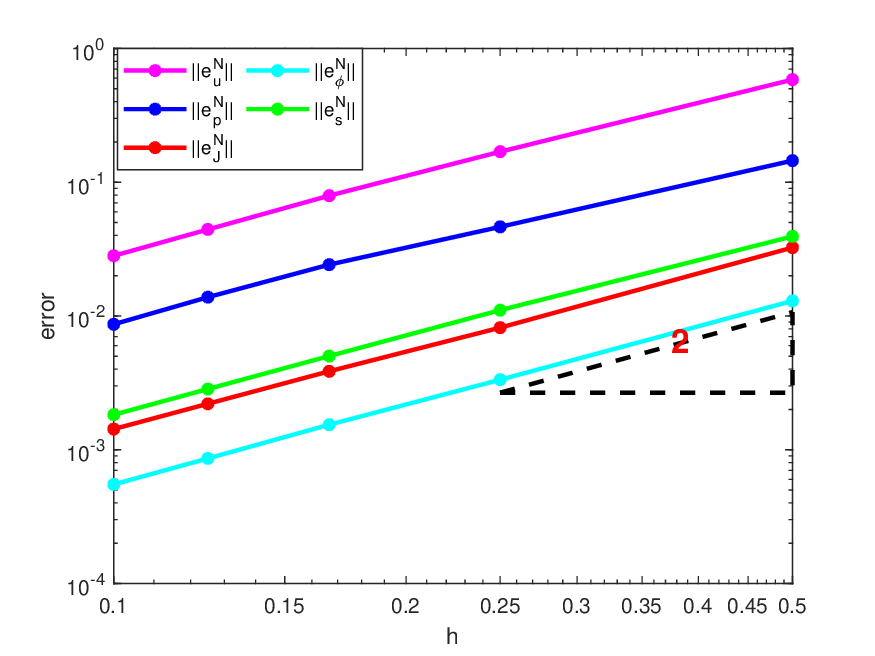}}
\hspace{-0.50cm}
\subfloat[\texttt{Simple-Voronoi}$-h_{V}$]{
\label{fig.mesh3dc_first_mixed}
\includegraphics[width=5.5cm]{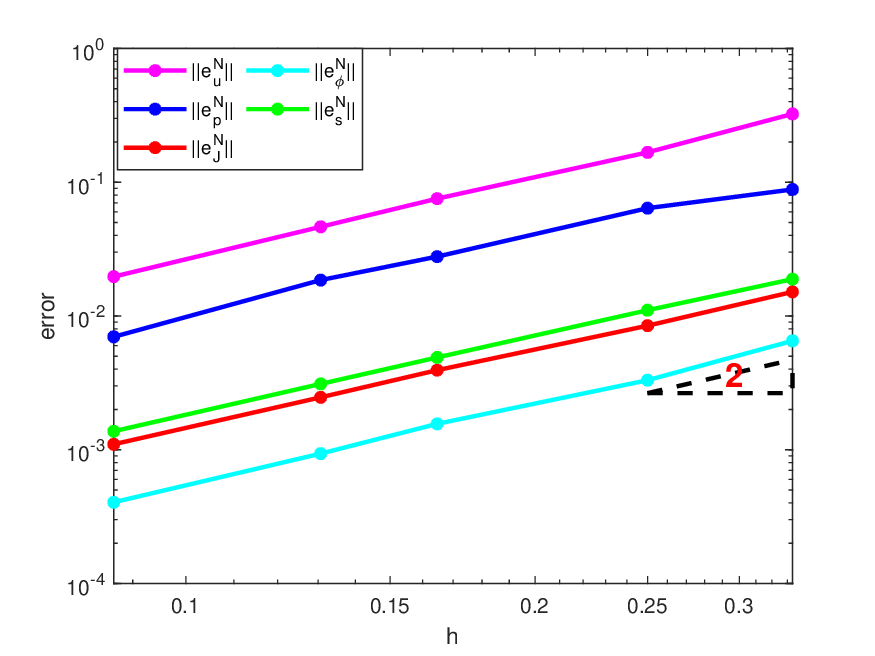}} 
\caption{Convergence rates of the first-order fully discretization at $t=1$, $\Delta t = h_{i}^{2}$.}
\label{the first-order fully discretization}
\end{figure}
\begin{table}[htp]\footnotesize
\centering
\setlength{\tabcolsep}{1mm}{
\caption{Divergence-free results for the first-order fully discretization at $t=1$, $\Delta t = h_{i}^{2}$.\label{table:Divergence-free-full-discrete-first}}
\begin{tabular*}{14cm}{@{\extracolsep{\fill}} c c c | c c c| c c c}
\hline
&\multicolumn{2}{c|}{\texttt{Dtp}}&&
\multicolumn{2}{c|}{\texttt{Cube}}&
\multicolumn{2}{c}{\texttt{Simple-Voronoi}} \\
\hline
 $h_{D}$ & $\|\text{div}\boldsymbol{u}_{h}^{N}\|_{0}$  & $\|\text{div}\boldsymbol{J}_{h}^{N}\|_{0}$ &$h_{C}$ &$\|\text{div}\boldsymbol{u}_{h}^{N}\|_{0}$& $\|\text{div}\boldsymbol{J}_{h}^{N}\|_{0}$&$h_{S}$ &$\|\text{div}\boldsymbol{u}_{h}^{N}\|_{0}$ &$\|\text{div}\boldsymbol{J}_{h}^{N}\|_{0}$\\
 \hline
3.9685e-01 & 9.0556e-17 & 1.2957e-16 & 5.0000e-01 & 9.1798e-17 & 4.6301e-17 & 3.3333e-01 & 1.0958e-16 & 1.7532e-16\\
1.9843e-01 & 2.2485e-16 & 3.5585e-16 & 2.5000e-01 & 1.5227e-16 & 2.0245e-16 & 2.5000e-01 & 1.4234e-16 & 3.6880e-16\\
1.3228e-01 & 3.1521e-16 & 5.7200e-16 & 1.6667e-01 & 2.3517e-16 & 3.8858e-16 & 1.6466e-01 & 2.6326e-16 & 1.0606e-15\\
9.9213e-02 & 4.3997e-16 & 8.1408e-16 & 1.2500e-01 & 3.3393e-16 & 5.5315e-16 & 1.3069e-01 & 3.2147e-16 & 1.8076e-15\\
7.9370e-02 & 5.5049e-16 & 1.0867e-15 & 1.0000e-01 & 4.1530e-16 & 7.2736e-16 & 8.6670e-02 & 4.6727e-16 & 4.1647e-15\\
\hline
\end{tabular*}}
\end{table}
\begin{figure}
\centering
\subfloat[\texttt{Dtp}$-h_{T}$]{
\label{fig.mesh3da_second_mixed}
\includegraphics[width=5.5cm]{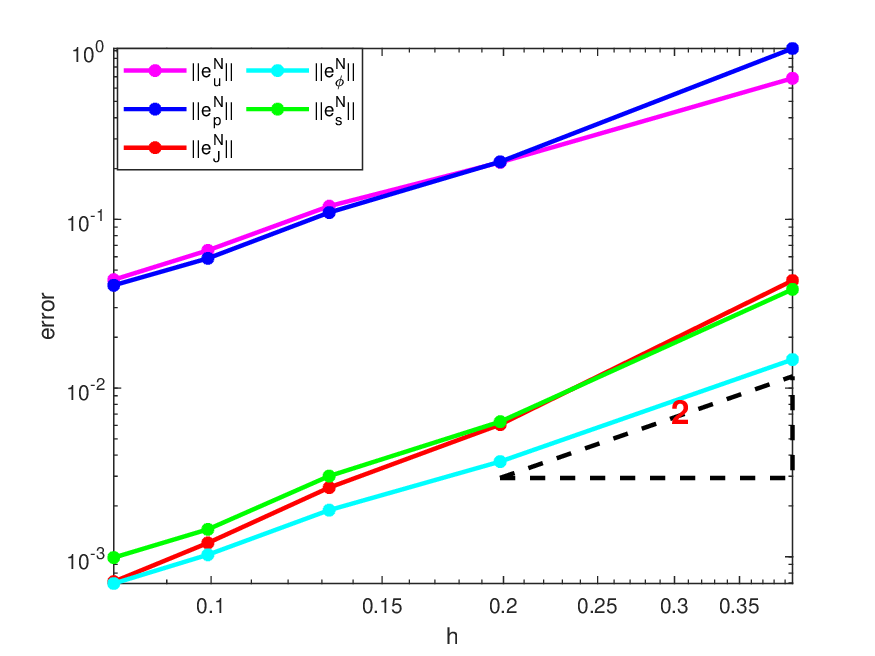}}
\hspace{-0.50cm}
\subfloat[\texttt{Cube}$-h_{C}$]{
\label{fig.mesh3db_first_mixed}
\includegraphics[width=5.5cm]{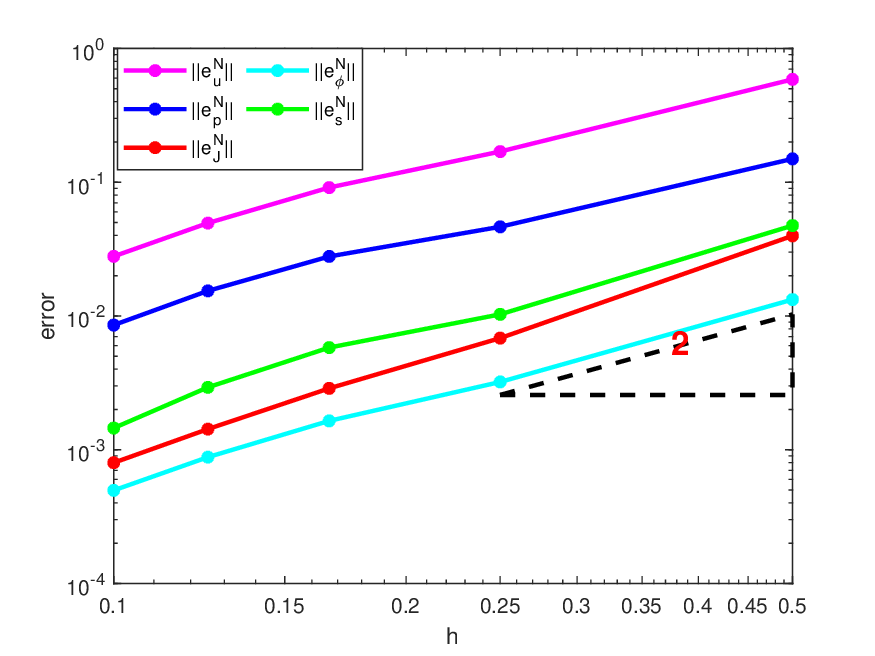}}
\hspace{-0.50cm}
\subfloat[\texttt{Simple-Voronoi}$-h_{V}$]{
\label{fig.mesh3dc_first_mixed}
\includegraphics[width=5.5cm]{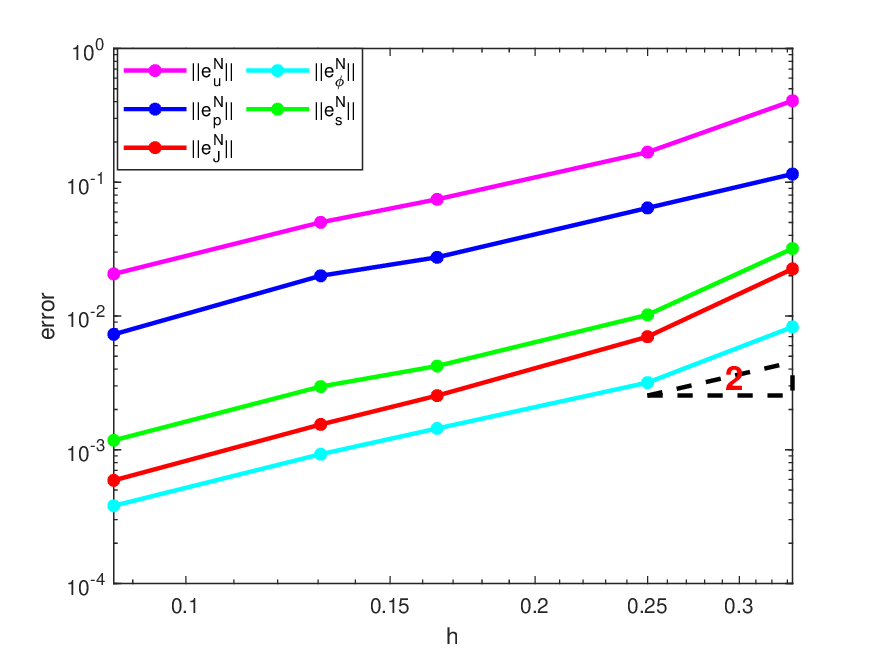}} 
\caption{Convergence rates of the second-order fully discretization at $t=1$, $\Delta t = h_{i}$.}
\label{the second-order fully discretization}
\end{figure}
\begin{table}[htp]\footnotesize
\centering
\setlength{\tabcolsep}{1mm}{
\caption{Divergence-free results for the second-order fully discretization at $t=1$, $\Delta t = h_{i}$.\label{table:Divergence-free-full-discrete-second}}
\begin{tabular*}{14cm}{@{\extracolsep{\fill}} c c c | c c c| c c c}
\hline
&\multicolumn{2}{c|}{\texttt{Dtp}}&&
\multicolumn{2}{c|}{\texttt{Cube}}&
\multicolumn{2}{c}{\texttt{Voronoi}} \\
\hline
 $h_{D}$ & $\|\text{div}\boldsymbol{u}_{h}^{N}\|_{0}$  & $\|\text{div}\boldsymbol{J}_{h}^{N}\|_{0}$ &$h_{C}$ &$\|\text{div}\boldsymbol{u}_{h}^{N}\|_{0}$& $\|\text{div}\boldsymbol{J}_{h}^{N}\|_{0}$&$h_{S}$ &$\|\text{div}\boldsymbol{u}_{h}^{N}\|_{0}$ &$\|\text{div}\boldsymbol{J}_{h}^{N}\|_{0}$\\
 \hline
3.9685e-01 & 1.1214e-16 & 1.3260e-16 & 5.0000e-01 & 6.1085e-17 & 4.8062e-17 &  3.3333e-01 & 6.1085e-17 & 1.8355e-16\\
1.9843e-01 & 2.2078e-16 & 3.5244e-16 & 2.5000e-01 & 1.5315e-16 & 1.9309e-16 &  2.5000e-01 & 1.7680e-16 & 3.0536e-16\\
1.3228e-01 & 3.4947e-16 & 6.0643e-16 & 1.6667e-01 & 3.0154e-16 & 4.2784e-16 &  1.6466e-01 & 2.2827e-16 & 9.3946e-16\\
9.9213e-02 & 4.4251e-16 & 8.0859e-16 & 1.2500e-01 & 3.7782e-16 & 6.0243e-16 &  1.3069e-01 & 3.2332e-16 & 1.8848e-15\\
7.9370e-02 & 5.8306e-16 & 1.0862e-15 & 1.0000e-01 & 4.0934e-16 & 7.1577e-16 &  8.6670e-02 & 4.5275e-16 & 4.1722e-15\\
\hline
\end{tabular*}}
\end{table}
\section{Conclusion}
This paper proposes a novel first-order and a novel second-order fully discrete virtual element schemes based on the scalar auxiliary variable method for the three dimensional inductionless magnetohydrodynamics problem. The backward Eular formula and the backward differential formula are used for the time discretization, and two types conservation virtual element formulations are employed for spatial discretization. Thus, the mass conservation in the
velocity field and the charge conservation in the current density field are kept by taking characteristics
of the virtual element method's discrete complex structures. Moreover, the nonlinear term is handled explicitly based on the SAV method. The current density field is decoupled from the momentum equation, and the velocity field is decoupled from Ohm's law. Thus, we only need to solve the Stokes and Mixed-Poisson subproblems with constant coefficients at each time step, which ensure the high efficiency and unconditional stability of schemes.
\bibliographystyle{plain}
\bibliography{ref}

\end{document}